\documentclass[11pt,letterpaper]{amsart}
\usepackage[active]{srcltx}
\usepackage{hyperref}
\usepackage[normalem]{ulem}

\usepackage[normalem]{ulem}
\usepackage{graphicx}
\usepackage{amssymb}
\usepackage{epstopdf}
\usepackage{amsmath}
\usepackage{amsthm}
\usepackage{amssymb}
\usepackage{latexsym}
\usepackage{mathtools}
\usepackage{mathrsfs}
\usepackage{graphics}
\usepackage{latexsym}
\usepackage{psfrag}
\usepackage{import}
\usepackage{verbatim}
\usepackage{enumerate}
\usepackage{enumitem}
\usepackage{graphicx}
\usepackage[usenames]{color}

\usepackage[capitalise]{cleveref}

\newcommand{\R}{\mathbb{R}}
\newcommand{\N}{\mathbb{N}}

\newcommand{\dist}{\text{\rm dist}}
\newcommand{\supp}{\text{\rm supp}}

\newcommand{\id}{\mathrm{Id}}

\newcommand{\ve}{\varepsilon}

\newcommand{\cI}{\mathcal{I}}

\newcommand{\T}{\mathcal{T}}

\renewcommand{\L}{\mathcal{L}}

\newcommand{\CD}{\mathsf{CD}}

\newcommand{\Geo}{{\rm Geo}}
\newcommand{\TGeo}{{\rm TGeo}}
\newcommand{\MCP}{\mathsf{MCP}}

\newcommand{\Ent}{{\rm Ent}}

\newcommand{\Dom}{{\rm Dom}}
\newcommand{\fs}{\mathfrak{s}}
\newcommand{\fc}{\mathfrak{c}}
\newcommand{\fI}{\mathfrak{I}}

\newcommand{\fa}{\mathfrak{a}}
\newcommand{\fb}{\mathfrak{b}}

\newcommand{\mui}{\mu_\infty}
\newcommand{\rmC}{{\mathrm C}}

\newcommand{\Cb}[1]{\rmC_b(#1)}

\newcommand{\Prob}{\mathcal P}

\newcommand{\BB}{\mathscr{B}}
\newcommand{\BorelSets}[1]{\BB(#1)}

\newcommand{\Ric}{{\rm Ric}}

\renewcommand{\L}{\mathcal{L}}

\newcommand{\vol}{\mathrm{Vol}}

\newcommand{\TMCP}{\mathsf{TMCP}}
\newcommand{\TCD}{\mathsf{TCD}}
\newcommand{\wTCD}{\mathsf{wTCD}}

\newcommand{\mm}{\mathfrak m}
\newcommand{\qq}{\mathfrak q}

\newcommand{\ee}{{\rm e}}
\newcommand{\QQ}{\mathfrak Q}

\newcommand{\sfd}{\mathsf d}

\newcommand{\forevery}{\text{for every }}

\theoremstyle{plain}
\newtheorem{lemma}{Lemma}[section]
\newtheorem{theorem}[lemma]{Theorem}

\newtheorem{proposition}[lemma]{Proposition}
\newtheorem{corollary}[lemma]{Corollary}

\newtheorem*{theorem*}{Theorem}
\newtheorem*{maintheorem*}{Main Theorem}

\theoremstyle{definition}

\newtheorem{definition}[lemma]{Definition}
\newtheorem*{definition*}{Definition}
\newtheorem{remark}[lemma]{Remark}
\newtheorem{example}[lemma]{Example}

\numberwithin{equation}{section}

\title[An isoperimetric-type inequality in Lorentzian spaces]{A sharp isoperimetric-type inequality for Lorentzian spaces satisfying timelike Ricci lower bounds}
\author{Fabio Cavalletti}
\address{Department of Mathematics, 
University of Milan, Milan (Italy)}
\email{fabio.cavalletti@unimi.it} 

\author{Andrea Mondino}
\address{Mathematical Institute,  University of Oxford  (UK)} 
\email{andrea.mondino@maths.ox.ac.uk}

\date{\today}     
                                                         
\begin{document}

\begin{abstract}
The paper establishes a sharp and rigid isoperimetric-type inequality in Lorentzian signature under the assumption of Ricci curvature bounded below in the timelike directions. The inequality is proved in the high generality of Lorentzian pre-length spaces satisfying timelike Ricci lower bounds in a synthetic sense via optimal transport, the so-called $\TCD^e_p(K,N)$ spaces.  The results are new already for smooth Lorentzian manifolds. Applications include an upper bound on the area of  achronal hypersurfaces inside the interior of a black hole (original already in Schwarzschild) and an upper bound on the area of achronal hypersurfaces in cosmological spacetimes.
\end{abstract}

\maketitle

\setcounter{tocdepth}{1}
\tableofcontents

\section{Introduction}

\subsection*{Brief account on the isoperimetric problem in Riemannian signature}

The isoperimetric problem is one of the most classical problems in Mathematics, having its roots in the Greek legend of Dido, queen of Carthage. In \emph{Riemannian} signature, it amounts to answer the following question: 
\begin{equation*}
\text{``What is the maximal volume that can be enclosed by a given area?"}
\end{equation*}
Equivalently, it can be stated as the problem of finding the \emph{maximal} function $I_{(M,g)}(\cdot):[0,\infty)\to [0,\infty)$ such that for every subset $E\subset M$ with smooth boundary $\partial E$ in the $(n+1)$-dimensional Riemannian manifold $(M^{n+1},g)$, then
\begin{equation}\label{eq:IsopIneqRiem}
{\rm Vol}_g^{n}(\partial E)\geq I_{(M,g)}({\rm Vol}_g^{n+1}(E)),
\end{equation}
where ${\rm Vol}_g^{n+1}(E)$ (resp.\;${\rm Vol}_g^{n}(\partial E)$) denotes the $(n+1)$-dimensional measure of $E$ with respect to $g$ (resp. the $n$-dimensional measure of $\partial E$ with respect to the restriction of $g$).

The literature on the isoperimetric problem in Riemannian signature is highly extensive (see, for instance, \cite{Osserman, Ros}). Even in Euclidean spaces, the complete solution is relatively recent and required several significant breakthroughs. In the broader context of sets with finite perimeter, a complete proof was established by De Giorgi \cite{DeGiorgi} (refer to \cite{DeGiorgi2} for the English translation), following a long series of noteworthy intermediate results. It is worth mentioning Steiner \cite{Steiner}, who introduced a now-classical symmetrization technique that now bears his name, with this objective in mind.  Quantitative versions of the Euclidean isoperimetric inequality have been established by Fusco-Maggi-Pratelli \cite{FuMP}, Figalli-Maggi-Pratelli \cite{FiMP}, Cicalese-Leonardi \cite{CiLe}.

In the framework of Riemannian manifolds with Ricci curvature bounded below, an isoperimetric inequality in the form \eqref{eq:IsopIneqRiem} (with the function $I_{(M,g)}(\cdot)$ here depending only on the dimension and on the Ricci lower bound) 
was proved by Gromov \cite{Gromov} in case of positive Ricci lower bound, 
following a previous work by L\'evy \cite{Levy}.
After several contributions,
the case of bounded diameter and a general Ricci lower bound 
was established in a sharp form by E.\;Milman \cite{MilmanJEMS}
to which we refer for a complete account on the bibliography.  
The L\'evy-Gromov-Milman isoperimetric inequality was extended to non-smooth metric measure spaces, together with an analysis of the rigidity and stability questions, by the authors \cite{CM1}; a quantitative version was obtained by the authors with Maggi \cite{CaMaMo}. Moreover, in the unbounded case of complete Riemannian manifolds with non-negative Ricci curvature and positive asymptotic volume ratio, a sharp isoperimetric inequality was proved by Agostiniani-Fogagnolo-Mazzieri \cite{AFM-Inv} in dimension three and by Brendle \cite{Brendle-CPAM} in any dimension. Their isoperimetric inequality was extended to non-smooth metric measure spaces by Balogh-Krist\'aly \cite{BK-MathAnn}; the analysis of the equality case with the associated rigidity was performed by Antonelli-Pasqualetto-Pozzetta-Semola \cite{APPS-MathAnn} and  Manini with the first named author \cite{CaMa-JEMS}.

Isoperimetric bounds of the type \eqref{eq:IsopIneqRiem} have proven to be extremely influential in mathematical general relativity, for initial data sets (hence still in  Riemannian signature). Let us mention the Riemannian Penrose inequality \cite{Huisken-Ilmanen, Bray:PenroseIneq}, and the concept of isoperimetric mass introduced by Huisken \cite{Huisken, HuiskenVideo}.

\subsection*{Brief account on the isoperimetric problem in Lorentzian signature}
If $(M^{n+1},g)$ is a \emph{Lorentzian} manifold, the \emph{maximal} function $I_{(M,g)}(\cdot):[0,\infty)\to [0,\infty)$ satisfying \eqref{eq:IsopIneqRiem} is identically $0$ - at least for small volumes and, in several examples (including Minkowski spacetime), for all volumes. The reason is that the causal diamonds have positive $(n+1)$-volume, but their boundary is a null hypersurface (with singularities of negligible measure) with zero $n$-volume, with respect to the restriction of the ambient Lorentzian metric.

Indeed, due to the different signature, a geometric \emph{minimization} problem in \emph{Riemannian} signature is turned into a \emph{maximization} problem in \emph{Lorentzian} signature. A landmark example is given by geodesics, which locally \emph{minimize} length in \emph{Riemannian} signature, and instead locally \emph{maximize} time separation (i.e.\;Lorentzian length) in \emph{Lorentzian} signature.  An analogous phenomenon appears for the isoperimetric problem which, in a first approximation, should be phrased as:
\begin{equation*}
\text{``What is the maximal  area that can be used to enclose a given volume?"}
\end{equation*}
Equivalently, it can be stated as the problem to find the \emph{minimal} function $J_{(M,g)}(\cdot):[0,\infty)\to [0,\infty)$ such that for every subset $E\subset M$ with smooth boundary $\partial E$ in the $(n+1)$-dimensional Lorentzian manifold $(M^{n+1},g)$, then
\begin{equation}\label{eq:IsopIneqLor}
{\rm Vol}_g^{n}(\partial E)\leq J_{(M,g)}({\rm Vol}_g^{n+1}(E)),
\end{equation}
where ${\rm Vol}_g^{n+1}(E)$ (resp.\;${\rm Vol}_g^{n}(\partial E)$) denotes the $(n+1)$-dimensional measure of $E$ with respect to $|g|$ (resp. the $n$-dimensional measure of $\partial E$ with respect to the restriction of $|g|$).

In sharp contrast to the Riemannian signature, where the literature on the isoperimetric problem is extensive, the literature on the isoperimetric problem in Lorentzian signature is more limited. To the best of our knowledge, the following are all the results in the literature at the moment:
\begin{itemize}
\item Bahn-Ehrlich \cite{BE99} in 1999 provided an upper bound on the area of a compact spacelike achronal hypersurface $S$ contained in the future of a point $O$ in the $(n+1)$-dimensional Minkowski spacetime, in terms of the volume (raised to the appropriate power to obtain scale-invariance) of a suitably constructed past cone $C(S)$ with base point $O$:
\begin{equation}\label{eq:IsopIneqMink}
{\rm Vol}_g^{n}(S) \leq C(n)\; ({\rm Vol}_g^{n+1}(C(S)))^{n/(n+1)}.
\end{equation}
\item Bahn \cite{Bahn99} in 1999 obtained a similar inequality to \cite{BE99} for two-dimensional Lorentzian surfaces with Gaussian curvature bounded from above.
\item Abedin-Corvino-Kapita-Wu \cite{ACKW09} in 2009 generalized the isoperimetric inequality of \cite{BE99} to spacetimes $I \times \mathbb{H}^{n}$ with a warped metric $g = -(dt)^{2} + a(t)^{2}g_{\mathbb{H}^{n}}$ satisfying $a''\leq 0 $. This corresponds to a subclass of Friedman-Robertson-Walker spacetimes satisfying the strong energy condition $\Ric(v,v)\geq 0$ for timelike vectors $v$.
\item Lambert-Scheuer \cite{LambSche} in 2021 extended \cite{ACKW09} to spacetimes $N = (a,b)\times S_{0}$ with metric $g = - dr^{2} + \theta(r)^{2} \hat{g}$ satisfying the null convergence condition and with $(S_{0},\hat{g})$ compact. The relation between the area and the volumes is as follows: if $\Sigma \subset N$ is a spacelike, compact, achronal, and connected hypersurface, then
$$
{\rm Vol}_g^{n}(\Sigma) \leq \varphi({\rm Vol}_g^{n+1}(C(\Sigma))),
$$
where $C(\Sigma)$ is the region between $\Sigma$ and ${a } \times S$ and $\varphi$ is the function which gives equality on the coordinate slices.
\end{itemize}

We also mention the work of Tsai-Wang \cite{TsaiWang} in 2020 which established an isoperimetric-type inequality for maximal, spacelike submanifolds in Minkowski space. With a slightly different perspective, Graf-Sormani in \cite{GrafSormani} have recently improved on Truede-Grant \cite{TreudeGrant} by establishing upper bounds on the Lorentzian area and volume of slices of the time-separation function from a Cauchy hypersurface in terms of its mean curvature.

Arguably, two of the main motivations for such a short bibliography 
compared with the Riemannian case are: 
\begin{itemize}
\item  the a-priori lack of  regularity for the isoperimetric problem for general domains, due to the failure of  ellipticity caused by the Lorentzian signature of the ambient metric; nevertheless, let us mention the deep work by Bartnik \cite{Bartnik-ACTA} on the regularity of variational maximal surfaces. 
We bypass the regularity issues by adopting an optimal transport approach which does not assume any regularity of the subsets;
\item there are deep geometric differences between the isoperimetric problem in Lorentzian and in Riemannian signature, in addition to the aforementioned maximization versus minimization nature. In the Riemannian case, the isoperimetric profile function $I_{(M^{n+1},g)}$ as in \eqref{eq:IsopIneqRiem} typically satisfies 
\begin{equation}\label{eq:IsopRiemVto0}
I_{(M^{n+1},g)}(v)=O(v^{n/(n+1)})\to 0,\quad \text{ as $v\to 0$},
\end{equation}
as it easily follows by choosing small geodesic spheres as competitors.
For instance, the Euclidean isoperimetric inequality states that
\begin{equation}\label{eq:IsopRn}
I_{\R^{n+1}}(v)= C(n)\,  v^{n/(n+1)}, \quad \text{for all }v>0.
\end{equation}
In sharp contrast, in Lorentzian signature, one cannot expect the profile function $J$ as in \eqref{eq:IsopIneqLor} to have a power-like expression as in \eqref{eq:IsopRn} for general subsets even in very symmetric spaces,  \emph{without further assumptions on the geometry of the subsets} (for instance, in  \eqref{eq:IsopIneqMink} only conical subsets in Minkowski spacetime are considered).
Even the mild asymptotic property \eqref{eq:IsopRiemVto0} for small volumes fails dramatically for general subsets. Indeed, in general, $J(v)$ \emph{does not tend to zero}, as $v\to 0$; see  Example \ref{Example:1} below.

Compared with the aforementioned Lorentzian references, our approach will be markedly different: instead of constraining the geometry of the subsets, we will consider an extra term in the Lorentzian isoperimetric inequality, keeping track of the ``time-separation" of the subset.
\end{itemize}

\begin{example}\label{Example:1}
    Let $(M^{n+1},g)$ be a globally hyperbolic spacetime admitting compact Cauchy hypersurfaces (this is not strictly needed, but it simplifies the discussion for the sake of an example).  Let $V\subset M$ be a spacelike compact Cauchy hypersurface, of positive $n$-dimensional volume (with respect to the  Riemannian metric given by the restriction of $g$ to $V$). Let $S_j$ be a sequence of spacelike compact Cauchy hypersurfaces in the future of $V$, and converging smoothly to $V$ as $j\to \infty$. Denote by $C(V, S_j)$ the region between $V$ and $S_j$. Then it is readily seen that 
$${\rm Vol}_g^{n}(S_j) \to {\rm Vol}_g^{n}(V)>0, \quad   {\rm Vol}_g^{n+1}(C(V,S_j))\to 0, \quad \text{as } j\to \infty.$$
In particular,  $J(v)\not\to 0$ as $v\to 0$.
\end{example}

\subsection*{Main results of the paper}

The results will be proved under very low regularity assumptions both on the spacetime and on the subsets, namely in the framework of Lorentzian pre-length spaces satisfying timelike Ricci curvature lower bounds in a synthetic sense via optimal transport, the so-called $\TCD^e_p(K,N)$ spaces. 
\\The setting of Lorentzian pre-length spaces was introduced by Kunzinger-S\"amann \cite{KS} building on the notion of causal spaces pioneered by Kronheimer-Penrose \cite{CausalSpace}. The framework comprises Lorentzian manifolds with metrics of low regularity; namely, locally Lipschitz Lorentzian metrics and, more generally, continuous causally plain Lorentzian metrics, studied by Chru\'sciel-Grant \cite{CG}.
An optimal transport characterization of Ricci curvature lower bounds in the timelike directions for smooth Lorentzian manifolds was obtained by McCann \cite{McCann} and by Mondino-Suhr \cite{MoSu}. The theory of  $\TCD^e_p(K,N)$ spaces has been developed by the authors of the present paper in \cite{CaMo:20}; see also the related work by Braun \cite{Braun} and the survey \cite{CaMo:22}. 
\smallskip

For the sake of the introduction, the statements will be presented in a simplified setting (both on the ambient spacetime and on the subsets in consideration), referring to the body of the paper for the more general case. Let us stress that the results seem to be original already in the simplified form below.
\smallskip

Let $(M^{n+1},g)$ be a smooth, globally hyperbolic,  Lorentzian manifold. Denote by $\ll$ the chronological relation: for $x,y\in M$, we say that $x\ll y$ if there exists a Lipschitz timelike curve from $x$ to $y$.
Let $\tau:M\times M\to [0,\infty]$ be the time-separation function on $M$ defined by
$$
\tau(x,y):=
\begin{cases} & \sup\{ L_g(\gamma) \mid \gamma:I\to M \; \text{ timelike Lipschitz curve}\}, \quad \text{if } x\ll y, \\
& 0 \qquad \text{otherwise},
\end{cases}
$$
where 
$$
{\rm L}_g(\gamma):=\int_I \sqrt{|g(\dot{\gamma}_t, \dot{\gamma}_t)|} \, {\rm d}t
$$
is the length of the timelike Lipschitz curve $\gamma:I\to M$.

The time-separation function satisfies the reverse triangle inequality (on timelike triples) and should be thought of as a Lorentzian counterpart of the distance function in Riemannian geometry.
 Given a subset $V\subset M$, we denote its chronological future by $I^{+}(V)$:
$$I^+(V):=\{y\in M\mid \exists x\in V \text{ such that } x\ll y\}.$$
The time-separation function from $V$ is defined by
$$
\tau_V:I^+(V)\to (0,\infty], \quad \tau_V(x):=\sup_{y\in V} \tau(y,x),
$$
and it should be thought as the Lorentzian distance from $V$. 
In the body of the paper $\tau_V$ will also be defined, adopting the analogous notation, over the chronological past of $V$ as well.

Let $V\subset M$ be a Cauchy hypersurface in $M$, and let $S$ be a compact acausal (i.e., no pair of points in $S$ is causally related) Borel subset contained in the chronological future of $V$, i.e.\;$S\subset I^+(V)$.

Define the ``distance" from $V$ to $S$ by
$$
\dist(V,S):=\inf_{x\in S} \tau_V(x)
$$
and consider the geodesically conical region from $V$ to $S$ defined by
\begin{equation*}
C(V,S):=\{\gamma_t\mid t\in [0,1], \text{ such that } \gamma_0\in V,\, \gamma_1\in S,\, L_g(\gamma)= \tau_{V}(\gamma_1)\}, 
\end{equation*}
i.e., $C(V,S)$ is the region spanned by timelike geodesics from $V$ to $S$, realizing $\tau_{V}$.
We can now state the main result of the paper (in a simplified form, for the sake of the introduction).

\begin{theorem}[A sharp isoperimetric-type inequality]\label{Thm:IsopIneqIntro}
Let $(M^{n+1},g)$ be a globally hyperbolic Lorentzian manifold satisfying Hawking-Penrose's strong energy condition (i.e., $\Ric\geq 0$ on timelike vectors) and  let $V \subset M$ be  Cauchy hypersurface. 
Then for any compact and acausal hypersurface $S\subset I^+(V)$  the following inequality is valid
$$
{\rm Vol}_g^n(S) \; \dist(V,S)\leq (n+1) \; {\rm Vol}_g^{n+1}(C(V,S)).
$$
\end{theorem}

\begin{remark}
\begin{itemize}
\item As we have reported, the existing literature about isoperimetric-type inequalities in Lorentzian manifolds (or in Riemannian spacelike slices) assumes the metric $g$ to be a warped product.  Recall that  global hyperbolicity implies the existence of a global time function and thus a product structure \emph{as a differentiable manifold}; however, let us stress that  there is no symmetry assumption on the Lorentzian metric in \cref{Thm:IsopIneqIntro}, but merely a lower bound on the Ricci curvature in the timelike directions.
\item \cref{Thm:IsopIneqIntro} is stated for non-negative Ricci curvature just for the sake of simplicity. A completely analogous statement holds for Ricci curvature bounded below by $K\in \R$ in the timelike directions. Also
the assumptions on $V$ and $S$ can be relaxed considerably: it is enough to assume that $V$ is a Borel, achronal, timelike complete subset and that 
\begin{itemize}
\item Either $S$ is a compact acausal Borel set, disjoint from the boundary of $X$ (see \cref{T:isop1} and \cref{cor:IsopSmooth} for the precise statements);
\item Or $S$ is a Borel, achronal (possibly unbounded) subset with empty future $V$-boundary (in the sense of \cref{D:V-boundary}), see \cref{R:otherinequalities}.
\end{itemize}

\item The isoperimetric-type inequality in \cref{Thm:IsopIneqIntro} is sharp (see \cref{prop:SharpnessIsop}) and rigid (see \cref{prop:RigidityIsop}): the equality is attained if and only if the spacetime is conical. 
\end{itemize}
\end{remark}

\noindent
As applications we will establish:
\begin{itemize}
    \item An upper bound on the area of  compact acausal  (or, possibly unbounded, Cauchy) hypersurfaces inside the interior of a black hole, see \cref{rem:AreaBoundCauchyBlackHole}. The bound seems to be new already in the interior of the Schwarzschild black hole, see \cref{Example:SchwInterior}.
    \item An upper bound on the area of compact acausal  (or, possibly unbounded, Cauchy) hypersurfaces in cosmological spacetimes. The novelty with respect to previous results (see for instance \cite{ACKW09, Flaim}) is that no symmetry is assumed; this higher generality seems to have advantages also for applications (see  for instance \cite{ESA}). We refer to \cref{rem:CauchyCosmological} for more details.
\end{itemize}
Let us also mention the next result, establishing a monotonicity formula for the area of the level sets of the distance function from a Cauchy hypersurface.

\begin{theorem}[Area Monotonicity]\label{T:introAreaMon}
Let $(M^{n+1},g)$ be a globally hyperbolic Lorentzian manifold satisfying Hawking-Penrose's strong energy condition (i.e. $\Ric\geq 0$ on timelike vectors). Let $V \subset M$ be a Cauchy hypersurface and let $V_{t}: = \{\tau_{V} = t\}$ be the achronal slice at distance $t>0$ from $V$.
Then the map
\begin{equation}\label{E:monotonicity1intro}
(0,\infty) \ni t \longmapsto \frac{{\rm Vol}_g^n(V_{t})}{t^{n-1}}
\end{equation}
is monotonically non-increasing.
\end{theorem}

\begin{remark}
\begin{itemize}
    \item  In the setting of CMC Einstein flows, a pointwise monotonicity formula 
similar to \eqref{E:monotonicity1intro}  goes back to \cite{FischerMoncrief} and \cite{Anderson}.
There  the spacelike hypersurfaces $\Sigma_{t}$ considered are constant mean curvature compact surfaces parametrized by the Hubble time 
$t = -n/H$. Such monotonicity has then been used to study the convergence as $t \to \infty$ of the metric; we refer to \cite{LottCollapsing}
for more details (see also \cite{LottInitial} for similar result when $t \to 0$).

\item
    \cref{T:introAreaMon} is stated for non-negative Ricci curvature just for the sake of simplicity. A completely analogous statement holds for Ricci curvature bounded below by $K\in \R$ in the timelike directions.
Also the assumption on $V$ can be relaxed considerably: it is enough to assume that $V$ is a Borel, achronal, timelike complete subset. For the general statement, refer to \cref{T:monotonicityVolume}.   

    \item The monotonicity formula for the area \eqref{E:monotonicity1intro} is sharp (see \cref{Rem:SharpMononot}): the equality is attained if and only if the spacetime is conical.
\end{itemize}
\end{remark}

\subsection*{Challenges and new ideas}
It is tempting to compare \cref{Thm:IsopIneqIntro} (and its proof) with the L\'evy-Gromov inequality and, in particular, its proof obtained by the authors \cite{CM1} for non-smooth metric measure spaces with lower Ricci bounds in a synthetic sense, the so-called $\CD^*(K,N)$ spaces. Indeed, both proofs are based on a dimension reduction argument, reducing the proof to $1$-dimensional problems thanks to a disintegration theorem (see \cref{T:disint}) and a localization of the curvature-dimension conditions (see \cref{T:local}). However, the Lorentzian signature poses serious challenges, yielding major differences in the proofs.

Indeed, a key point in the proof of the L\'evy-Gromov inequality in $\CD^*(K,N)$ spaces obtained in \cite{CM1}, was to perform an $L^1$-optimal transportation of mass from a Borel set $E$ of finite measure to its complement (or, more precisely, between their normalized characteristic measures). Such an $L^1$-optimal transportation induces a partition of the ambient space into geodesics (up to a set of measure zero), called \emph{rays}; moreover, thanks to the structure of $L^1$-optimal transportation, all the intersections of such rays with $E$ have the same measure (with respect to suitable weighted 1-dimensional measures, induced by the partition via the disintegration theorem).
This allows to show a lower bound on the perimeter of the Borel set $E$, by reducing the proof to the one dimensional problems obtained by intersecting $E$ with the optimal transport rays.

A major challenge to implement such a strategy in the Lorentzian setting is that one would need a timelike optimal transport from $E$ to its complement. Except from extremely symmetric situations (e.g., slabs in product spaces) it is highly unclear how general such an assumption would be; actually, it is immediate to build examples of sets $E$ where such a timelike optimal transport to the complement does not exist.
\\

Motivated by the above discussion, in the present paper, we adopt a different approach. 
Let $V$ be a Cauchy hypersurface  (actually, $V$ Borel achronal, timelike complete would suffice) in the ambient (possibly non-smooth) Lorentzian space $X$. Consider the partition of $I^+(V)$ induced by the flow lines of $\nabla \tau_V$, called \emph{$V$-rays}; this can be made rigorous even without differentiability assumptions on $\tau_V$, by studying the induced transport relation (see \cref{Ss:transportrelation}). By construction, such $V$-rays are Lorentzian geodesics. One can use the  disintegration theorem (see \cref{T:disint}) and localize the curvature-dimension conditions $\TCD^e_p(K,N)$ (see \cref{T:local}). Up to here, i.e., up to the end of \cref{S:localization}, the paper is an improvement of our previous \cite{CaMo:20}, where we localized the weaker $\TMCP(K,N)$ along the same partition.

The real technical novelty  of the present work is \cref{S:isoperimetric}, where we define and study a Lorentzian analog of the Minkowski content, subordinated to $\tau_V$, for Borel sets $A\subset I^+(V)$. In sharp contrast with the metric setting, where the sub-level sets of the distance function from a compact subset are compact and thus with finite measure, in the Lorentzian signature the sub-level sets $\{\tau_V\leq C\}$ are a-priori non-compact and with infinite measure. This poses serious technical challenges. To overcome these, we introduce new ideas:   to handle the case of possibly unbounded subsets $A$, we introduce the condition of \emph{empty future $V$-boundary} (see \cref{D:V-boundary}), we show that it is satisfied by Cauchy hypersurfaces (see \cref{L:Cauchyboundary}) and that it allows a reduction to 1-dimensional estimates of the Lorentzian Minkowski content (see \cref{P:main1} and \cref{P:identityslice}). In order to simplify the statements, and motivated by applications, in \cref{SSec:MinkBounded} we study the case when $A\subset I^+(V)$ is a compact acausal subset and we obtain the same 1-dimension reduction as above, under the assumption that $A$ does not intersect the boundary of $X$ (see \cref{P:main1_bounded},  \cref{C:ineqCompactAcausal} and \cref{rem:AI-B} for the precise -- and more general -- statements).

Building on top of the 1-dimensional reduction of the Lorentzian Minkowski content obtained in \cref{S:isoperimetric} and the 1-dimensional localization of the curvature-dimension condition established in \cref{T:local}, we prove the new Lorentzian isoperimetric-type inequality  in \cref{T:isop1}. This is the main result of the work. When compared with the L\'evy-Gromov inequality, a novelty in the statement is the appearance of the ``Lorentzian distance" from $V$ to $S$, denoted by $\dist(V,S)$. The reason for such a difference is both geometric (one cannot hope to prove an exact Lorentzian counterpart of the L\'evy-Gromov inequality, due to spaces such as in Example \ref{Example:1}) and technical, at the level of the proof -- as already discussed above.

\subsection*{Note added after completion}
The main results of the present work, with their proofs, were presented at various conferences since March 2023. After completion of  a first version of the preprint in fall 2023, we were informed of the independent work \cite{BraunMcCann} by Braun-McCann where the authors develop a synthetic framework for variable timelike Ricci lower bounds. Though the papers have different scopes and disjoint main results, there is non-empty overlapping about some of the techniques, in particular on some of the content of  \cref{S:localization}.

\subsection*{ Acknowledgments}  
A.\,M.\;acknowledges support from the European Research Council (ERC) under the European Union's Horizon 2020 research and innovation programme, grant agreement No.\;802689 ``CURVATURE''.  For the purpose of Open Access, he has applied a CC BY public copyright licence to any Author Accepted Manuscript (AAM) version arising from this submission.

Part of this research was carried out at the Fields Institute of Toronto, during the Thematic Program  ``Nonsmooth Riemannian and Lorentzian Geometry'', and at the Erwin Schr\"odinger International Institute for Mathematics and Physics (ESI) in Vienna, during the workshop ``Non-regular Spacetime Geometry''. The authors wish to express their appreciation to the two institutions and to the organisers of the corresponding events for the stimulating atmosphere and the excellent working conditions.

\section{Preliminaries}\label{S:Basics}

\subsection{Lorentzian length spaces}
 \label{Ss:Lorentzlength}

Following the work of Kunzinger-S\"amann \cite{KS},
in this section we briefly recall some basic notions 
from the theory of Lorentzian 
(pre-)length spaces.
We refer to \cite{KS} and to \cite{CaMo:20} for further details and for the proofs.

\begin{definition}[Causal space,  Kronheimer-Penrose \cite{CausalSpace}]
A \emph{causal space}  $(X,\ll,\leq)$ is a set $X$ endowed with a preorder $\leq$ and a transitive relation $\ll$ contained in $\leq$.
\end{definition}

We write $x<y$ when $x\leq y, x\neq y$. We say that $x$ and $y$ are timelike (resp. causally) related if $x\ll y$  (resp. $x\leq y$).  Let $A\subset X$ be an arbitrary subset of $X$. We define the chronological (resp. causal) future of $A$ the set
\begin{align*}
I^{+}(A)&:=\{y\in X\,:\, \exists x\in A \text{ such that } x\ll y\},\\
J^{+}(A) &:=\{y\in X\,:\, \exists x\in A \text{ such that } x\leq y\},
\end{align*}
respectively. 
Analogously, we define the   chronological (resp. causal) past of $A$. 
In case $A=\{x\}$ is a singleton, with a slight abuse of notation, 
we will write $I^{\pm}(x)$ (resp. $J^{\pm}(x)$) instead of  $I^{\pm}(\{x\})$ (resp. $J^{\pm}(\{x\})$).
Moreover 
we also introduce the following notations 
$$
X^{2}_{\leq}:=\{(x,y) \in X\times X\,:\, x\leq y \} , \qquad X^{2}_{\ll}:=\{(x,y) \in X\times X\,:\, x\ll y \}.
$$

\begin{definition}[Lorentzian pre-length space $(X,\sfd, \ll, \leq, \tau)$]
A \emph{Lorentzian pre-length space} $(X,\sfd, \ll, \leq, \tau)$ is a  causal space $(X,\ll,\leq)$ additionally  equipped with a proper metric $\sfd$  (i.e. closed and bounded subsets are compact) and a lower semicontinuous function $\tau: X\times X\to [0,\infty]$,  called \emph{time-separation function}, satisfying
\begin{equation}\label{eq:deftau}
\begin{split}
\tau(x,y)+\tau(y,z)\leq \tau (x,z) &\quad\forall x\leq y\leq z \quad \text{reverse triangle inequality} \\
\tau(x,y)=0, \; \text{if } x\not\leq y, & \quad  \tau(x,y)>0 \Leftrightarrow x\ll y.
\end{split}
\end{equation}
\end{definition}

\noindent
The lower semicontinuity of $\tau$ implies that $I^{\pm}(x)$ is open, for any $x\in X$.\\
The set $X$ is endowed with the metric topology induced by $\sfd$. All the topological concepts on $X$ will be formulated in terms of such metric topology.
For instance, we will denote by $\overline{C}$ the topological closure (with respect to $\sfd$) of a subset $C\subset X$.

\begin{definition}[Causal/timelike curves]
If $I\subset \R$ is an interval, a non-constant curve $\gamma:I\to X$ is called (future-directed) \emph{timelike} (resp. \emph{causal}) if $\gamma$ is locally Lipschitz continuous (with respect to $\sfd$) and if for all $t_{1}, t_{2}\in I$, with $t_{1}<t_{2}$, then $\gamma_{t_{1}}\ll \gamma_{t_{2}}$ (resp. $\gamma_{t_{1}}\leq \gamma_{t_{2}}$). We say that $\gamma$ is a \emph{null} curve if, in addition to being causal, no two points on $\gamma(I)$ are related with respect to $\ll$. 
\end{definition}

\smallskip
The length of a causal curve is defined via the time separation function, in analogy to the theory of length  metric spaces:
for $\gamma:[a,b]\to X$ future-directed causal we set 
\begin{equation*}
{\rm L}_{\tau}(\gamma):=\inf\left\{ \sum_{i=0}^{N-1} \tau(\gamma_{t_{i}}, \gamma_{t_{i+1}})  \,:\, a=t_{0}<t_{1}<\ldots<t_{N}=b, \; N\in \N \right\}.
\end{equation*}
In case the interval is half-open, say $I=[a,b)$, then the infimum is taken over all partitions with $a=t_{0}<t_{1}<\ldots<t_{N}<b$ (and analogously for the other cases).

Under fairly general assumptions, these definitions coincide with the classical ones in the smooth setting, 
see 
\cite[Prop.\;2.32, Prop.\;5.9]{KS}.

A future-directed causal curve $\gamma:[a,b]\to X$ is \emph{maximal} if  it realises the
time separation, i.e. if ${\rm L}_{\tau}(\gamma)=\tau(\gamma_{a}, \gamma_{b})$.
\\In case the time separation function is continuous with $\tau(x,x)=0$ for every $x\in X$, then any maximal timelike curve $\gamma$ with finite $\tau$-length has a (continuous, monotonically strictly increasing, not necessarily Lipschitz) 
reparametrization $\lambda$ by $\tau$-arc-length, i.e.
$\tau( \gamma_{\lambda(s_{1})}, \gamma_{\lambda(s_{2})})=s_{2}-s_{1}$ for all $s_{2}\leq s_{1}$ in the corresponding interval (see \cite[Cor.\;3.35]{KS}).

We therefore adopt the following convention:
a curve $\gamma$ 
will be called \emph{(causal) geodesic} if it is maximal and continuous 
when parametrized by $\tau$-arc-length. In other words, the set of (causal) geodesics is 
\begin{equation}\label{E:geodesic}
 \Geo(X):=\{ \gamma\in C([0,1], X)\,:  \, \tau(\gamma_{s}, \gamma_{t})=(t-s)\, \tau(\gamma_{0}, \gamma_{1})\, \forall s<t\}.   
\end{equation}
The set of \emph{timelike geodesic} is 
defined as follows: 
\begin{equation}\label{E:timegeo}
\TGeo(X):=\{ \gamma \in \Geo(X):  \, \tau(\gamma_{0}, \gamma_{1})>0\}.
\end{equation}
Given $x\leq y\in X$ we set
\begin{align}
\Geo(x,y)&:=\{ \gamma\in \Geo(X)\,:\, \gamma_{0}=x, \, \gamma_{1}=y\}  \label{eq:defGeo(x,y)} \\
\fI(x,y,t)&:=\{\gamma_{t}\,:\, \gamma\in \Geo(x,y)\}  \label{eq:defI(x,y,t)} 
\end{align}
respectively the space of geodesics, and the set of $t$-intermediate points  from $x$ to $y$.
\\ If $x\ll y\in X$, we call 
$$\TGeo(x,y):=\{ \gamma\in \TGeo(X)\,:\, \gamma_{0}=x, \, \gamma_{1}=y\}. $$
Given two subsets $A,B\subset X$, we denote
\begin{equation}\label{eq:defI(A,B,t)}
\fI(A,B,t):=\bigcup_{x\in A, y\in B} \, \fI(x,y,t) 
\end{equation}
the subset of $t$-intermediate points of geodesics from points in $A$ to points in $B$.

\begin{definition}[Timelike non-branching]\label{def:TNB}
A  Lorentzian pre-length space $(X,\sfd, \ll, \leq, \tau)$ is said to be \emph{forward timelike non-branching}  if and only if for any $\gamma^{1},\gamma^{2} \in \TGeo(X)$, the following holds:
$$
\exists \;  \bar t\in (0,1) \text{ such that } \ \forall t \in [0, \bar t\,] \quad  \gamma_{ t}^{1} = \gamma_{t}^{2}   
\quad 
\Longrightarrow 
\quad 
\gamma^{1}_{s} = \gamma^{2}_{s}, \quad \forall s \in [0,1].
$$
It is said to be \emph{backward timelike non-branching} if the reversed causal structure is forward timelike non-branching. In case it is both forward and backward timelike non-branching it is said \emph{timelike non-branching}.
\end{definition}  

By Cauchy Theorem, it is clear that if $(M,g)$ is a spacetime whose Christoffel symbols are locally-Lipschitz (e.g. in case $g\in C^{1,1}$) then the associated synthetic structure is timelike non-branching. 
For spacetimes with a metric of lower regularity (e.g. $g\in C^{1}$ or $g\in C^{0}$) timelike branching may occur.

\subsubsection{Causal Ladder}

Concerning the causal ladder, we 
follow \cite{Minguzzi:23}. In order 
to streamline the presentation we 
will only consider \emph{Lorentzian geodesic spaces,} i.e. 
 Lorentzian pre-length spaces $(X,\sfd, \ll, \leq, \tau)$ that  additionally are: 
\begin{itemize}
\item \emph{$\sfd$-Compatible:} every $x\in X$ admits a neighbourhood $U$ and a constant $C$ such that $L_{\sfd}(\gamma)\leq C$ for every causal curve $\gamma$ contained in $U$;
\item  \emph{Geodesic:} for all $x,y\in X$ with $x<y$ there is a future-directed causal curve $\gamma$ from $x$ to $y$ with $\tau(x,y)= {\rm L}_{\tau}(\gamma)$.
\end{itemize}
A Lorentzian geodesic space is in particular a Lorentzian length space, see  \cite[Def.\;3.22]{KS}. 

Hence from \cite[Cor.\;3.8]{Minguzzi:23} we can consider the following  version of global hyperbolicity that fits with the previous literature. 
A Lorentzian geodesic space $(X,\sfd, \ll, \leq,\tau)$ is called
\begin{itemize}
\item \emph{Causal}: if $\leq$  is also antisymmetric, i.e. $\leq$ is a partial order; 
\item \emph{Globally hyperbolic}: if it is causal and for every $x,y\in X$ the causal diamond $J^{+}(x)\cap J^{-}(y)$ is compact in $X$.
\end{itemize}

From \cite[Thm.\;3.7]{Minguzzi:23} this definition of global hyperbolicity is equivalent with the one adopted in \cite{KS} (that we omit). 
Also global hyperbolicity implies that the relation $\leq$ is a closed subset of 
$X\times X$.
It was proved in \cite[Thm.\;3.28]{KS} that for a globally hyperbolic  Lorentzian geodesic  space $(X,\sfd, \ll, \leq,\tau)$, the time-separation function $\tau$ is finite and continuous: in particular 
the previous remark on the existence of constant $\tau$-speed parametrizations for maximal causal curves applies, thus any two distinct causally related points are joined by a causal geodesic. 
\\From \cite{Minguzzi:23}
it also follows that if  $X$ is globally hyperbolic and $K_{1}, K_{2}\Subset X$ are compact subsets  then 
$$
\fI(K_{1},K_{2},t)\Subset \bigcup_{t\in [0,1]}  \fI(K_{1},K_{2},t) \Subset X, \quad \forall t\in [0,1].
$$

%

%

%

\subsection{Optimal transport in Lorentzian geodesic spaces}
\label{Ss:OTLorentz}

We start by briefly recalling some notation about convergence of probability measures.

Given a complete and separable (in particular, everything hold for proper) metric space $(X,\sfd)$, we denote by
$\BorelSets X$ the collection of all Borel subsets of $X$ and 
by $\mathcal P(X)$ (resp.  $\mathcal{P}_{c}(X)$)  the  collection of all Borel probability
measures (resp. with compact support).
We say that $(\mu_{k})\subset \mathcal P(X)$ \emph{narrowly converges} to $\mu_{\infty}\in \mathcal P(X)$ if
\begin{equation}\label{eq:defNarrowConv}
\lim_{k\to\infty}\int f\,\mu_k=\int f\,\mui\qquad\forevery f\in \Cb X
\end{equation}
where $\Cb X$ denotes the space of bounded and continuous functions.

We next review some basics on optimal transport
in the Lorentzian synthetic setting. 
For simplicity of presentation, we will assume that
$(X,\sfd, \ll, \leq,\tau)$
is a globally hyperbolic  Lorentzian geodesic space; 
we refer to \cite{CaMo:20} for more general results.

Given $\mu,\nu\in \mathcal{P}(X)$,  the set of \emph{transport plans} is 
$$\Pi(\mu,\nu):=\{\pi\in  \mathcal{P}(X\times X) \,:\, (P_{1})_{\sharp}\pi=\mu, \, (P_{2})_{\sharp}\pi=\nu \}.$$
The set of \emph{causal} and \emph{timelike transport plans} are defined by
\begin{align*}
 \Pi_{\leq}(\mu,\nu)&:=\{\pi\in  \Pi(\mu,\nu) \,:\,  \pi(X^{2}_{\leq})=1 \}, \\
  \Pi_{\ll}(\mu,\nu)&:=\{\pi\in  \Pi(\mu,\nu) \,:\,  \pi(X^{2}_{\ll})=1 \}. 
\end{align*}
As
$X^{2}_{\leq}\subset X^{2}$ is a closed subset, 
$\pi\in  \Pi_{\leq}(\mu,\nu)$ if and only if $\supp\, \pi \subset X^{2}_{\leq}$.
For $p\in (0,1]$, given $\mu,\nu\in \mathcal{P}(X)$, the $p$-Lorentz-Wasserstein distance is defined by
\begin{equation}\label{eq:defWp}
\ell_{p}(\mu,\nu):= \sup_{\pi \in \Pi_{\leq}(\mu,\nu)} \left(  \int_{X\times X}  \tau(x,y)^{p} \, \pi(dxdy)\right)^{1/p}.
\end{equation}
If $\Pi_{\leq}(\mu,\nu)=\emptyset$ we set $\ell_{p}(\mu,\nu):=-\infty$.
A plan  $\pi\in  \Pi_{\leq}(\mu,\nu)$ maximising in \eqref{eq:defWp} is said \emph{$\ell_{p}$-optimal}. The set of \emph{$\ell_{p}$-optimal} plans from $\mu$ to $\nu$ is denoted by $  \Pi_{\leq}^{p\text{-opt}}(\mu,\nu)$.

An alternative formulation of \eqref{eq:defWp} can be obtained by 
using the following function: 
\begin{equation}\label{eq:defell}
\ell(x,y)^{p}:=
\begin{cases}
\tau(x,y)^{p} \quad &  \text{if } x\leq y \\
-\infty \quad & \text{otherwise}.
\end{cases}
\end{equation}
Clearly, if $\pi\in \Pi_{\leq}(\mu,\nu)$, then $\pi$-a.e. one has $\tau(x,y)= \ell(x,y)$.  
Moreover, using the convention that $\infty-\infty=-\infty$, if $\pi\in \Pi(\mu,\nu)$ satisfies $ \int_{X\times X}  \ell(x,y)^{p} \, \pi(dxdy)>-\infty$ then $\pi\in \Pi_{\leq}(\mu,\nu)$.
Thus the maximization problem  \eqref{eq:defWp} is equivalent (i.e. the $\sup$ and the set of maximisers coincide) to the  maximisation problem
\begin{equation}\label{eq:supell}
\sup_{\pi \in \Pi(\mu,\nu)} \left(  \int_{X\times X}  \ell(x,y)^{p} \, \pi(dxdy)\right)^{1/p}.
\end{equation}
The advantage of the formulation \eqref{eq:supell} is that $\ell^p$ is upper semi-continuous on $X\times X$.
One can therefore invoke standard optimal transport techniques (e.g.\;\cite{villani:oldandnew}) to 
ensure the existence of a solution of the Monge-Kantorovich problem \eqref{eq:supell}. 

\begin{proposition}\label{prop:ExMaxellp}
Let  $(X,\sfd, \ll, \leq, \tau)$ be a  globally hyperbolic Lorentzian geodesic space and let $\mu,\nu\in \mathcal{P}(X)$. 
If  $ \Pi_{\leq}(\mu,\nu)\neq \emptyset$ and if  there exist measurable functions $a,b:X\to \R$, with $a\oplus b \in L^{1}(\mu\otimes \nu)$ such that $\ell^{p}\leq a\oplus b$  on $\supp \, \mu \times \supp \, \nu$ (e.g. when $\mu$ and $\nu$ are compactly supported) then the $\sup$ in \eqref{eq:defWp} is attained and finite. 
\end{proposition}

In Proposition \ref{prop:ExMaxellp} we used the following standard notation: given $\mu,\nu\in \Prob(X)$,  $\mu \otimes \nu\in \Prob(X^{2})$ is the product measure; given $u,v:X\to \R\cup \{+\infty\} $, the function $u\oplus v: X^{2}\to \R\cup \{+\infty\}$ is defined by $u\oplus v(x,y):=u(x)+v(y)$.

The Lorentzian-Wasserstein distance $\ell_{p}$ satisfies the reverse triangle inequality:
\begin{equation}\label{eq:RTIellq}
\ell_{p}(\mu_{0},\mu_{1})+ \ell_{p}(\mu_{1},\mu_{2})
\leq \ell_{p}(\mu_{0}, \mu_{2}), 
\quad \forall \mu_{0},\mu_{1},\mu_{2}\in \mathcal{P}(X),
\end{equation}
where we adopt the convention that $\infty-\infty=-\infty$ to interpret the left hand side of \eqref{eq:RTIellq}.

We also recall two relevant notions of cyclical monotonicity. 
\begin{definition}[$\tau^{p}$-cyclical monotonicity and $\ell^{p}$-cyclical monotonicity]\label{D:monotonicity}
A subset $\Gamma\subset X^{2}_{\leq}$ is said to be $\tau^{p}$-cyclically monotone (resp. $\ell^{p}$-cyclically monotone) if, for any $N\in \N$ and any family $(x_{1}, y_{1}), \ldots, (x_{N}, y_{N})$ of points in $\Gamma$,  it satisfies
\begin{equation}\label{eq:taupcyclmon}
\sum_{i=1}^{N}\tau(x_{i}, y_{i})^{p} \geq \sum_{i=1}^{N}\tau(x_{i+1}, y_{i})^{p},
\end{equation}
(resp. $ \sum_{i=1}^{N}\ell^{p}(x_{i}, y_{i}) \geq \sum_{i=1}^{N}\ell^{p}(x_{i+1}, y_{i}))$ with the convention $x_{N+1}=x_{1}$. 
\end{definition}
Accordingly, a transport plan $\pi$
 is said to be $\tau^{p}$-cyclically monotone (resp. $\ell^{p}$-cyclically monotone) if there exists a 
 $\tau^{p}$-cyclically monotone set (resp. $\ell^{p}$-cyclically monotone set) $\Gamma$ such that $\pi(\Gamma) = 1$.
It is straightforward to check that 
$\tau^{p}$-cyclical monotonicity implies $\ell^{p}$-cyclical monotonicity. 
Moreover,  denoting by $P_i : X\times X \to X$ the projection on the $i$-th component,
if $P_{1}(\Gamma)\times P_{2}(\Gamma) \subset X^{2}_{\leq}$ then $\ell^{p}$-cyclical monotonicity is equivalent to $\tau^{p}$-cyclical monotonicity

Under fairly general assumptions, cyclical monotonicity and optimality are equivalent. Indeed:
\begin{proposition}[Prop.\;2.8 and Thm\;2.26 in \cite{CaMo:20}]\label{prop:cicmon<->opt}
If $(X,\sfd, \ll, \leq, \tau)$ is a  globally hyperbolic Lorentzian geodesic space, $p\in (0,1]$, $\mu,\nu\in \mathcal{P}(X)$  with $\ell_{p}(\mu,\nu) \in (0,\infty)$, then  for any $\pi\in \Pi_{\leq} (\mu,\nu)$ the following holds:
\begin{enumerate}
\item If $\pi$ is  $\ell_{p}$-optimal then $\pi$ is  $\ell^{p}$-cyclically monotone. 
\item If $\pi(X^{2}_{\ll}) = 1$ and $\pi$ is $\ell^{p}$-cyclically monotone then $\pi$ is $\ell_{p}$-optimal. 
\item If $\pi$ is  $\tau^{p}$-cyclically monotone then $\pi$ is $\ell_{p}$-optimal.
\end{enumerate}
\end{proposition}
\noindent
Finally, we recall from \cite{CaMo:20} the definition of (strongly) timelike $p$-dualisable probability measures.

\begin{definition}[(Strongly) Timelike $p$-dualisable measures]\label{D:dualisable}
 Let  $(X,\sfd, \ll, \leq, \tau)$ be a Lorentzian pre-length space and let $p\in (0,1]$. We say that $(\mu,\nu)\in \mathcal{P}(X)^{2}$ is \emph{timelike $p$-dualisable (by $\pi\in \Pi_{\ll}(\mu,\nu)$)}  if 
 \begin{enumerate}
\item  $\ell_{p}(\mu,\nu)\in (0,\infty)$;
\item  $\pi\in  \Pi_{\leq}^{p\text{-opt}}(\mu,\nu)$ and $\pi(X^{2}_{\ll})=1$;
\item there exist measurable functions $a,b:X\to \R$, with $a\oplus b \in L^{1}(\mu\otimes \nu)$ such that  $\ell^{p}\leq a\oplus b$ on $\supp \, \mu \times  \supp \, \nu $.
\end{enumerate}
We say that $(\mu,\nu)\in \mathcal{P}(X)^{2}$ is \emph{strongly timelike $p$-dualisable} if, in addition: 
\begin{enumerate}
\item [4.]  there exists a measurable $\ell^{p}$-cyclically monotone set $\Gamma\subset X^{2}_{\ll} \cap (\supp \, \mu \times \supp \,\nu)$ such that a coupling  $\pi\in \Pi_{\leq}(\mu,\nu)$ is $\ell_{p}$-optimal if  and only if $\pi$ is concentrated on $\Gamma$, i.e. $\pi(\Gamma)=1$.
\end{enumerate}
 \end{definition} 
 
The above notions are connected with the validity of Kantorovich duality (see \cite[Sec.\;2.4]{CaMo:20}). 
The notion of strongly timelike $p$-dualisability is non-vacuous:

\begin{lemma}[Cor.\;2.29 in \cite{CaMo:20}]\label{lem:q-dualPcX}
Fix $p\in (0,1]$. Let  $(X,\sfd, \ll, \leq, \tau)$ be a  globally hyperbolic Lorentzian geodesic space and let $\mu,\nu\in \mathcal{P}(X)$ satisfy that:
\begin{enumerate}
\item there exist measurable functions $a,b:X\to \R$  with $a\oplus b \in L^{1}(\mu\otimes \nu)$ such that $\tau^{p}\leq a \oplus b$ on $\supp \, \mu \times  \supp \, \nu $;
\item $\supp \, \mu \times \supp \, \nu \subset X^{2}_{\ll}$.
\end{enumerate}
Then $(\mu,\nu)$  is strongly timelike $p$-dualisable.
\end{lemma}

\subsubsection{Geodesics of probability measures in the Lorentz-Wasserstein space}
\label{Ss:geodesicstructure}

Let us start by introducing some classical notation.
The evaluation map is defined by 
\begin{equation}\label{def:eet}
\ee_{t}: C([0,1], X) \to X, \quad \gamma\mapsto \ee_{t}(\gamma):=\gamma_{t}, \quad \forall t\in [0,1].
\end{equation}
The stretching/restriction operator ${\rm restr}_{s_{1}}^{s_{2}}: C([0,1], X) \to C([0,1], X)$ is defined by
\begin{equation}\label{def:restr}
({\rm restr}_{s_{1}}^{s_{2}} \gamma)_{t}:= \gamma_{(1-t)s_{1}+t s_{2}}, \quad \forall s_{1},s_{2}\in [0,1], s_{1}<s_{2}, \, \forall t\in [0,1].
\end{equation}

\begin{definition}[$\ell_p$-optimal dynamical plans and $\ell_p$-geodesics]\label{def:ellp-DOP}
Let  $(X,\sfd, \ll, \leq, \tau)$ be a Lorentzian pre-length space and let $p\in (0,1]$. We say that $\eta\in \mathcal{P}(\Geo(X))$ is an  \emph{$\ell_p$-optimal dynamical plan} from $\mu_0\in \mathcal{P}(X)$ to $\mu_1\in \mathcal{P}(X)$ if $(\ee_0)_\sharp \eta=\mu_0, (\ee_1)_\sharp \eta=\mu_1$ and 
\begin{equation}\label{eq:defODP}
(\ee_0, \ee_1)_{\sharp} \eta \quad \text{ belongs to }\Pi^{p\text{-opt}}_{\leq} ((\ee_0)_\sharp \eta, (\ee_1)_\sharp \eta).
\end{equation}
The set of $\ell_p$-optimal dynamical plans from $\mu_{0}$ to $\mu_{1}$ is denoted by ${\rm OptGeo}_{\ell_{p}}(\mu_{0}, \mu_{1})$.
We say that a curve $[0,1] \ni t \mapsto  \mu_{t} \in \mathcal{P}(X)$ 
is an $\ell_{p}$-geodesic if there exists an $\ell_p$-optimal dynamical plan $\eta$ from $\mu_0$ to $\mu_1$ such that $\mu_t=(\ee_t)_\sharp \eta$, for all $t\in [0,1]$.
\end{definition}
Notice that if  $\eta\in {\rm OptGeo}_{\ell_{p}}(\mu_{0}, \mu_{1})$, then the $\ell_p$-geodesic 
$$
\mu_t:=(\ee_t)_\sharp \eta, \quad \forall t\in [0,1],
$$
is continuous in narrow topology and satisfies
$
\ell_{p}(\mu_{s}, \mu_{t})=(t-s) \ell_{p}(\mu_{0}, \mu_{1})$, for all $s ,t  \in [0,1]$.

Let us recall that if $\mu_{0},\mu_{1}\in \Prob(X)$ have compact support, then there always exists an  $\ell_p$-optimal dynamical plan $\eta\in {\rm OptGeo}_{\ell_{p}}(\mu_{0}, \mu_{1})$ (and thus an $\ell_{p}$-geodesic) from $\mu_{0}$ to $\mu_{1}$, see \cite[Prop.\;2.33]{CaMo:20} for the proof and for other properties of $\ell_p$-optimal dynamical plans.

\subsubsection{Time-separation functions from sets and their transport relations}
\label{Ss:transportrelation}
A subset $V\subset X$ is called \emph{achronal} if $x\not \ll y$ for every $x,y\in V$. In particular, if $V$ is achronal, then $I^{+}(V)\cap I^{-}(V)= \emptyset$, so we can define the \emph{signed time-separation} to $V$, $\tau_{V}:X\to [-\infty, +\infty]$, by
\begin{equation}\label{eq:deftauV}
\tau_{V}(x):=
\begin{cases}
\sup_{y\in V} \tau(y,x), &\quad \text{ for }x\in I^{+}(V)\\
-\sup_{y\in V} \tau(x,y),& \quad \text{ for }x\in I^{-}(V) \\
0 &\quad \text{ otherwise}
\end{cases}.
\end{equation}
Note that $\tau_{V}$ is lower semi-continuous on $I^{+}(V)$ as supremum of continuous functions, 
and is upper semi-continuous on  $I^{-}(V)$.
\\In order for these suprema to be attained, global hyperbolicity and geodesic property of $X$ alone are not sufficient. One should rather demand additional compactness properties of the set $V$. The following notion, introduced by Galloway \cite{Ga} in the smooth setting, is well suited to this aim.

\begin{definition}[Future timelike complete (FTC) subsets]\label{def:FTC}
A subset $V\subset X$ is \emph{future timelike complete} (FTC), if for each point  $x\in I^{+}(V)$, the intersection $J^{-}(x)\cap V \subset V$ has compact closure (w.r.t. $\sfd$) in $V$. Analogously, one defines \emph{past timelike completeness} (PTC). A subset that is both   FTC and PTC is called \emph{timelike complete}.
\end{definition}

\begin{lemma}[Lemma 4.1, \cite{CaMo:20}]\label{L:initialpoint} 
Let $(X,\sfd, \ll, \leq, \tau)$ be a globally hyperbolic Lorentzian geodesic  space and let $V\subset X$ be an achronal FTC (resp. PTC) subset. Then  for each $x\in I^{+}(V)$ (resp. $x\in I^{-}(V)$) there exists a point $y_{x}\in V$ with $\tau_{V}(y_{x})=\tau(y_{x},x)>0$ (resp. $\tau_{V}(y_{x})=-\tau(x,y_{x})<0$).

Moreover for all $x,z\in I^{+}(V)\cup V$,
\begin{equation}\label{eq:tauvzxtau}
\tau_{V}(z) - \tau_{V}(x) \geq \tau(y_{x},z)-\tau(y_{x},x)  \geq  \tau(x,z), 
\end{equation}
provided $(x,z) \in X^{2}_{\leq}$. An analogous statement is valid for 
$x\in I^{-}(V)$.
\end{lemma}

By considering a non-smooth analogue of the gradient flow lines of $\tau_V$, one can obtain a partition into timelike geodesics of the future of $V$ (up to a set of measure zero). We briefly review this construction and refer to \cite[Sec.\;4.1]{CaMo:20} for more details.
\\First, notice that \eqref{eq:tauvzxtau} can be  extended to the whole 
$X^{2}$ by replacing $\tau$ with $\ell$, defined in \eqref{eq:defell}: 
\begin{equation}\label{E:triang}
\tau_{V}(z) - \tau_{V}(x) \geq
\ell(x,z), \qquad \forall x,z \in (I^{+}(V)\cup I^{-}(V)\cup V)^{2}.
\end{equation}
For ease of writing, we will use the following notation 
$$
I^{\pm}(V) 
: = (I^{+}(V)\cup I^{-}(V)\cup V).
$$
We associate to $V$ the following set:
\begin{equation}\label{E:GammaV}
\begin{split}
\Gamma_{V} : = &~ \{ (x,z) \in I^{\pm}(V)^{2} \cap X^{2}_{\leq} \, \colon \, 
 \tau_{V}(z) - \tau_{V}(x)  = \tau(x,z)>0 \} \\
 & \quad \cup \{(x,x) \,:\, x\in I^{\pm}(V)\}.
 \end{split}
\end{equation}
From the 
inequality \eqref{E:triang}, it follows straightforwardly that 
the set $\Gamma_{V}$ is $\ell$-cyclically monotone.
This implies the well-known
alignment along geodesics of the pairs belonging to $\Gamma_{V}$: 
for instance 
if $(x,z) \in \Gamma_{V}$ with $x \neq z$ and $x \in I^{+}(V)$, 
there exist $y \in V, \gamma \in \TGeo(y,z)$ and $t \in (0,1)$ such that 
$$
x = \gamma_{t}, \qquad \tau(y,\gamma_{s}) = \tau_{V}(\gamma_{s})  \quad \forall s\in [0,1], \qquad
(\gamma_{s},\gamma_{t}) \in \Gamma_{V}  \quad \forall s\in [0,t].
$$
An analogous property holds true if $z \in I^{-}(V)$. Again for all the details we refer to \cite{CaMo:20}.
Next we set $\Gamma_{V}^{-1}:=\{(x,y)\,:\, (y,x)\in \Gamma_{V}\}$ and we consider the \emph{transport relation} $R_{V}$ and 
the \emph{transport set with endpoints} $\T_{V}^{end}$
\begin{equation}\label{E:transport}
R_{V} : = \Gamma_{V} \cup \Gamma_{V}^{-1}, \qquad 
\T_{V}^{end} : = P_{1}(R_{V}\setminus \{ x = y \}),
\end{equation}
where $P_1$ denotes the projection on the first coordinate.
The transport relation will be an equivalence relation on a suitable subset of 
$\T_{V}^{end}$, constructed below.
Define the following subsets:
\begin{equation}\label{eq:defendpoints}
\begin{split}
\fa(\T_{V}^{end}) : =&~ \{ x \in \T_{V}^{end} \colon \nexists y \in \T_{V}^{end} \ s.t. \ (y,x) \in \Gamma_{V}, y\neq x \} \\
\fb(\T_{V}^{end}) : =&~ \{ x \in \T_{V}^{end} \colon \nexists y \in \T_{V}^{end} \ s.t. \ (x,y) \in \Gamma_{V}, y\neq x \},
\end{split}
\end{equation}
called the set of \emph{initial} and \emph{final points}, respectively.
Define the \emph{transport set without endpoints} 
\begin{equation}\label{E:nbtransport}
\T_{V} : = \T_{V}^{end} \setminus (\fa(\T_{V}^{end}) \cup \fb(\T_{V}^{end})).
\end{equation}

\begin{lemma}[Lemma 4.4, \cite{CaMo:20}] \label{lem:I+VTV}
If $V\subset X$ is a Borel
achronal timelike complete subset, 
then:
$$ 
I^{+}(V) \cup I^{-}(V)  = \mathcal{T}_{V}^{end}    \setminus V.
$$
\end{lemma}

If additionally $X$ is assumed to be timelike (backward and forward) non-branching,
then the transport relation $R_{V}$ 
is an equivalence relation over $\T_{V}$. The next lemma gives a clear description of the equivalences classes.

\begin{lemma}\label{lem:XalphaI}
For each equivalence class $[x]$ of $(\T_{V,}R_{V})$ there exists a  convex set  $I\subset \R$ of the real line  and a bijective map $F:I\to [x]$ satisfying:
\begin{equation}\label{eq:FIsometry}
\tau(F(t_{1}), F(t_{2})) = t_{2}-t_{1}, \quad \forall \,t_{1}\leq t_{2} \in I.
\end{equation}
Moreover, calling $\overline{\{z\in [x]\}}$ the topological closure of $\{z\in [x]\}\subset X$, then 
\begin{equation}\label{eq:closureVSendpoints}
\overline{\{z\in [x]\}}\setminus \{z\in [x]\} = \overline{\{z\in [x]\}}\setminus \T_{V}  \subset \fa(\T_{V}^{e}) \cup \fb(\T_{V}^{e}).
\end{equation}
\end{lemma}

The equivalence classes of 
$R_{V}$ inside $\mathcal T_{V}$ will be called \emph{V-rays} (or transport rays).

Concerning the measurability properties of the sets we have considered so far, 
the set $I^{+}(x) = \{ y \in X \colon \tau(x,y) >0 \}$
is open  by lower semi-continuity  of $\tau$ and the same is valid 
for $I^{-}(x)$.  
Accordingly, $I^{+}(V)=\bigcup_{x\in V}  I^{+}(x)$ is an open subset of $X$, 
and the same is valid for $I^{-}(V)$.
Since
$\tau_{V}$ is $\sup$ of continuous functions,  it is lower semi-continuous.
It follows that the set $\Gamma_{V}$ is Borel measurable (see \eqref{E:GammaV}).
It follows that also  $R_{V}$ is  Borel measurable, yielding  that $\T_{V}^{end}$ defined in \eqref{E:transport} is an analytic set 
(recall that analytic sets are precisely projections of Borel subsets of complete and separable metric spaces and the $\sigma$-algebra they generate is denoted by $\mathcal{A}$, we refer to \cite{Srivastava} for more details).
The transport set $\T_{V}$ defined in \eqref{E:nbtransport} 
can be proved to be an analytic set as well \cite[Lem.\;4.7]{CaMo:20}.

In order to induce a non-trivial measure theoretic decomposition from a partition (see the \Cref{T:disint} below), it is necessary to show the existence of a \emph{measurable selection} of a representative from each equivalence class.
Indeed, the existence of
an $\mathcal{A}$-measurable quotient map $\QQ$ of the equivalence relation $R_{V}$ over $\T_{V}$ can be obtained by a careful use of selection theorems:

\begin{lemma}[Prop.\;4.9, \cite{CaMo:20}]\label{lem:Qlevelset}
There exists an $\mathcal{A}$-measurable quotient map 
$\QQ:\T_{V}\to X$ of the equivalence relation $R_{V}$ over $\T_{V}$, 
  i.e. 
\begin{equation}\label{E:quotient}
\QQ : \mathcal{T}_{V} \to \mathcal{T}_{V}, 
\quad (x,\QQ(x)) \in R_{V}, 
\quad (x,y) \in R_{V} \Rightarrow \QQ(x) = \QQ(y).
\end{equation}
\end{lemma}

We will denote by $Q:=\QQ(\T_{V}) \subset X$ the quotient set (which is $\mathcal{A}$-measurable),  and by $X_{\alpha}$, with $\alpha\in Q$, the $V$-rays. 
Recall that each $X_\alpha$ is isometric to a real, possibly unbounded, open interval.

We refer to \cite[Sect.\;4.1, 4.2]{CaMo:20} for the missing details.



\subsection{Synthetic Timelike Ricci curvature lower bounds}\label{S:TCD}

We briefly recall the synthetic formulation of timelike Ricci lower bounds  for a  globally hyperbolic measured Lorentzian geodesic space 
$(X,\sfd, \mm, \ll, \leq, \tau)$ as given in \cite{CaMo:20} (after \cite{McCann} and \cite{MoSu}), see also \cite{CaMo:22} for a survey.

\begin{definition}[Measured Lorentzian pre-length space $(X,\sfd,  \mm, \ll, \leq, \tau)$]
A \emph{measured Lorentzian pre-length space} $(X,\sfd, \mm, \ll, \leq, \tau)$ is a Lorentzian pre-length space endowed with a Radon non-negative measure $\mm$. We say that $(X,\sfd, \mm, \ll, \leq, \tau)$ is globally hyperbolic (resp. geodesic) if $(X,\sfd, \ll, \leq, \tau)$ is so.
\end{definition}

We denote $\mathcal{P}_{ac}(X)$ (resp. $\mathcal{P}_{c}(X)$) the space of probability measures absolutely continuous with respect to $\mm$ (resp. the space of probability measures with compact support).
Given $\mu \in \mathcal{P}(X)$
its relative entropy w.r.t. $\mm$ is given by
$$
\Ent(\mu|\mm) = \int_{M} \rho \log\rho \; \mm,
$$
if $\mu = \rho \, \mm$ is absolutely continuous with respect to $\mm$ and $(\rho\log \rho)_{+}$ is  $\mm$-integrable. 
Otherwise we set $\Ent(\mu|\mm) = +\infty$.

A simple application of Jensen inequality using the convexity of $(0,\infty)\ni t\mapsto t \log t$ gives 
\begin{equation}\label{E:jensenE}
\Ent(\mu|\mm)\geq -\log \mm(\supp \, \mu)>-\infty,\quad \forall \mu\in  \mathcal{P}_{c}(X).
\end{equation}
We set  $\Dom(\Ent(\cdot|\mm)):=\{\mu\in \mathcal{P}(X)\,:\, \Ent(\mu|\mm)\in \R\}$ to be the finiteness domain of the entropy.
An important property of the relative entropy is the lower-semicontinuity under narrow convergence.

The following is the definition of 
the synthetic timelike Ricci curvature lower bounds.

\begin{definition}[$\mathsf{TCD}^{e}_{p}(K,N)$ and  $\mathsf{wTCD}^{e}_{p}(K,N)$ conditions]\label{def:TCD(KN)}
Fix $p\in (0,1)$, $K\in \R$, $N\in (0,\infty)$. We say that  a  measured Lorentzian pre-length space $(X,\sfd,\mm, \ll, \leq, \tau)$ satisfies  $\mathsf{TCD}^{e}_{p}(K,N)$ 
(resp. $\mathsf{wTCD}^{e}_{p}(K,N)$)  if the following holds.
For any pair $(\mu_{0},\mu_{1})\in (\Dom(\Ent(\cdot|\mm)))^{2}$ which is   timelike $p$-dualisable   
(resp. $(\mu_{0},\mu_{1})\in [\Dom(\Ent(\cdot|\mm))\cap  \Prob_{c}(X)]^{2}$  
which is   strongly timelike $p$-dualisable) by some 
$\pi\in \Pi^{p\text{-opt}}_{\ll}(\mu_{0},\mu_{1})$,  
there exists an  $\ell_{p}$-geodesic $(\mu_{t})_{t\in [0,1]}$ such that  
the function $[0,1]\ni t\mapsto e(t) : = \Ent(\mu_{t}|\vol_{g})$ is 
semi-convex (and thus in particular it is locally Lipschitz in $(0,1)$) and it satisfies
\begin{equation}\label{eq:conveKN}
e''(t) - \frac{1}{N} e'(t)^{2 } \geq K \int_{X\times X} \tau(x,y)^{2} \, \pi(dxdy),
\end{equation}
in the distributional sense on $[0,1]$.
\end{definition}

Definition \ref{def:TCD(KN)} corresponds to a differential/infinitesimal formulation of the $\mathsf{TCD}^{e}_{p}(K,N)$ condition. In order to have also an integral/global  formulation it is convenient  to introduce  the following entropy (cf. \cite{EKS}) 
\begin{equation}\label{eq:defSN}
U_{N}(\mu|\mm) : = \exp\left(-\frac{\Ent(\mu|\mm)}{N} \right).
\end{equation}
It is straightforward to check that  $[0,1]\ni t\mapsto e(t)$ is semi-convex and  satisfies \eqref{eq:conveKN} if and only if  $[0,1]\ni t\mapsto u_{N}(t):= \exp(- e(t)/N)$ is semi-concave and satisfies
\begin{equation}\label{eq:sN''}
u_{N}''\leq -\frac{K}{N} \|\tau\|^{2}_{L^{2}(\pi)} \,  u_{N}.
\end{equation} 
Set
\begin{equation}\label{eq:deffsfc}
\fs_{\kappa}(\vartheta):=
\begin{cases}
\frac{1}{\sqrt{\kappa}} \sin(\sqrt{\kappa} \vartheta),   & \kappa>0\\
\vartheta, &\kappa=0\\
\frac{1}{\sqrt{-\kappa}} \sinh(\sqrt{-\kappa} \vartheta),   &\kappa<0\\
\end{cases}, \qquad 
\fc_{\kappa}(\vartheta):=
\begin{cases}
\cos(\sqrt{\kappa} \vartheta),   & \kappa\geq 0\\
\cosh(\sqrt{-\kappa} \vartheta),   &\kappa<0\\
\end{cases}, 
\end{equation}
and
\begin{equation}\label{eq:sigmakappa}
\sigma_{\kappa}^{(t)}(\vartheta):=
\begin{cases}
\frac{\fs_{\kappa}(t\vartheta)}{\fs_{\kappa}(\vartheta)}, \quad & \kappa\vartheta^{2}\neq 0 \text{ and } \kappa\vartheta^{2}<\pi^{2} \\
t,\quad & \kappa \vartheta^{2}=0\\
+\infty, \quad &  \kappa \vartheta^{2}\geq \pi^{2}
\end{cases}.
\end{equation}
Note that the function $\kappa\mapsto \sigma_{\kappa}^{(t)}(\vartheta)$ is non-decreasing for every fixed $\vartheta, t$.
With the above notation, the differential inequality \eqref{eq:sN''}  is equivalent to the integrated version (cf. \cite[Lemma 2.2]{EKS}):
\begin{equation}\label{eq:sNconc}
u_{N}(t) \geq \sigma^{(1-t)}_{K/N} \left(\|\tau\|_{L^{2}(\pi)}\right) u_{N}(0) + \sigma^{(t)}_{K/N} \left( \|\tau\|_{L^{2}(\pi)} \right) u_{N}(1).
\end{equation} 
We thus have that both the 
$\mathsf{TCD}^{e}_{p}(K,N)$  and the 
$\mathsf{wTCD}^{e}_{p}(K,N)$ 
can be formulated as 
in terms of \eqref{eq:sNconc}. 
It has recently been shown in \cite[Thm.\;3.35]{Braun} that under the timelike non-branching assumption
$\TCD^e_p$ and $\wTCD^e_p$ are equivalent conditions for any choice of the parameters $K$ and $N$.

By considering $(K,N)$-convexity properties only of those $\ell_{p}$-geodesics 
$(\mu_{t})_{t\in [0,1]}$ where $\mu_{1}$ is a Dirac delta one obtains the following weaker condition \cite{CaMo:20} (see also Sturm \cite{sturm:II} and Ohta \cite{Ohta1}) independent on $p$. 

\begin{definition}
Fix $K\in \R$, $N\in (0,\infty)$. The measured globally hyperbolic Lorentz geodesic space $(X,\sfd, \mm, \ll, \leq, \tau)$ satisfies $\mathsf{TMCP}^{e}(K,N)$ if and only if  for any $\mu_{0}\in \Prob_{c}(X)\cap \Dom(\Ent(\cdot|\mm))$ and for any $x_{1}\in X$ such that  $x\ll x_{1}$ for $\mu_{0}$-a.e. $x\in X$, there exists an $\ell_{p}$-geodesic $(\mu_{t})_{t\in [0,1]}$ from $\mu_{0}$ to  $\mu_{1}=\delta_{x_{1}}$
such that 
\begin{equation}\label{eq:defTMCP(KN)}
U_{N}(\mu_{t}|\mm) \geq \sigma^{(1-t)}_{K/N} \left( \|\tau(\cdot,x_1) \|_{L^2(\mu_0)}\right)\, 
U_{N}(\mu_{0}|\mm), \quad \forall t\in [0,1).
\end{equation}
\end{definition}

As expected, the $\mathsf{wTCD}^{e}_{p}(K,N)$ condition  implies the 
$\mathsf{TMCP}^{e}(K,N)$, see \cite[Prop.\;3.12]{CaMo:20}.
We next recall some useful results 
concerning the existence and uniqueness of optimal plans 
in timelike non-branching $\mathsf{TMCP}^{e}(K,N)$ spaces.

\begin{theorem}[Prop.\;3.19, Thm.\;3.20 and 3.21 in \cite{CaMo:20}]\label{T:1}
Let  $(X,\sfd, \mm, \ll, \leq, \tau)$ be a timelike non-branching, globally hyperbolic Lorentzian geodesic space satisfying $\mathsf{TMCP}^{e}(K,N)$.
\\Let $\mu_{0},\mu_{1}\in \Prob_{c}(X)$, with $\mu_{0}\in \Dom(\Ent(\cdot|\mm))$. Assume that there exists $\pi\in \Pi^{p\text{-opt}}_{\leq}(\mu_{0},\mu_{1})$ such that $ \pi \left( \{\tau>0\} \right)=1$.

Then there exists a unique optimal coupling $\pi\in \Pi^{p\text{-opt}}_{\leq}(\mu_{0},\mu_{1})$ such that $ \pi \left( \{\tau>0\} \right)=1$.  Such a coupling $\pi$ is induced by a map $T$, i.e.,  $\pi=(\id, T)_{\sharp} \mu_{0}$ and 
$$
\ell_{p}(\mu_{0}, \mu_{1} )^{p} =\int_{X} \tau(x,T(x))^{p} \, \mu_{0}(dx).
$$
Moreover there exists a unique $\eta\in  {\rm OptGeo}_{\ell_{p}}(\mu_{0},\mu_1)$  with $ (\ee_{0}, \ee_{1})_{\sharp} \eta \, ( \{\tau>0\})=1$ and such $\eta$ is induced by a map, i.e. there exists ${\mathfrak T}:X\to \TGeo(X)$ such that $\eta={\mathfrak T}_{\sharp} \mu_{0}$; in particular, 
$(\ee_{0}, \ee_{1})_{\sharp} \eta = \pi$.
Finally, the  $\ell_{p}$-geodesic  
$\mu_{t} = (\ee_{t})_{\sharp} \eta$
satisfies
$\mu_{t} = \rho_{t} \mm \ll \mm$.
\end{theorem}


\subsubsection{Disintegration of $\mm$ and regularity of conditional measures}
\label{Ss:disintegrationregularity}

The partition in $V$-rays recalled in \cref{Ss:transportrelation} has a natural interplay with the synthetic curvature conditions: via Disintegration Theorem (after Lemma \ref{lem:Qlevelset}) one can  associate to the partition of the transport set a decomposition in conditional measures of the reference measure $\mm$ that inherits the synthetic curvature-dimension properties.

Below, we briefly summarise the results from \cite[Sect.\;4]{CaMo:20}.
We will denote by  $\mathcal{M}_{+}(X)$ the 
space of non-negative Radon measures over $(X,\sfd)$. 
 Moreover, given a Lorentzian pre-length space $(X,\sfd, \ll, \leq, \tau)$, its \emph{causally-reversed structure} $(X,\sfd, \tilde{\ll}, \tilde{\leq}, \tilde{\tau})$ is obtained by setting
$$x\, \tilde{\ll}\, y \Leftrightarrow y \ll x,  \quad x\, \tilde{\leq}\, y \Leftrightarrow y \leq x,  \quad \tilde{\tau}(x,y):=\tau(y,x).$$
In the smooth setting, this procedure corresponds to reverse the sign of the time-orienting vector field. For smooth Lorentzian manifolds, a lower bound on the Ricci curvature for \emph{future-pointing} timelike vectors immediately implies the same lower Ricci bound on any  timelike vector; however this is not the case in the present synthetic setting (for instance, one could consider non-reversible Lorentz-Finsler structures).

\begin{theorem}\label{T:disint}
Let  $(X,\sfd, \mm, \ll, \leq, \tau)$ be a globally hyperbolic timelike non-branching Lorentzian geodesic  space satisfying $\mathsf{TMCP}^{e}(K,N)$, 
assume that the causally-reversed structure satisfies the same conditions and 
let $V\subset X$ be a Borel achronal timelike complete subset. 

Considering $\T_{V}^{end}, \fa(\T_{V}^{end}), \fb(\T_{V}^{end})$ and $\T_{V}$ defined in \eqref{E:transport}, \eqref{eq:defendpoints}, \eqref{E:nbtransport},
then $\mm(\fa(\T_{V}^{end}))=\mm(\fb(\T_{V}^{end})=0$ and the following disintegration formula is valid: 
\begin{equation}\label{E:disintegration}
\mm\llcorner_{\T^{end}_{V}} = 
\mm\llcorner_{\T_{V}} 
= \int_{Q} \mm_{\alpha}\, \qq(d\alpha)
\end{equation}
where $\qq$ is a Borel probability measure over $Q \subset X$ such that 
$\QQ_{\sharp}( \mm\llcorner_{\T_{V}} ) \ll \qq$ and the map 
$Q \ni \alpha \mapsto \mm_{\alpha} \in \mathcal{M}_{+}(X)$ satisfies the following properties:
\begin{itemize}
\item[(1)] for any $\mm$-measurable set $B$, the map $\alpha \mapsto \mm_{\alpha}(B)$ is $\qq$-measurable; \smallskip
\item[(2)] for $\qq$-a.e. $\alpha \in Q$, $\mm_{\alpha}$ is concentrated on $\QQ^{-1}(\alpha) = X_{\alpha}$ (strong consistency); \smallskip
\item[(3)] for $\qq$-a.e. $\alpha \in Q$,  
$\mm_{\alpha}\ll  \L^{1}\llcorner_{X_{\alpha}}$;

\item[(4)] writing
$\mm_{\alpha} = h(\alpha,\cdot) \L^{1}\llcorner_{X_{\alpha}}$, then for $\qq$-a.e. $\alpha\in Q$ it holds that
$h(\alpha,\cdot)\in L^{1}_{loc}(X_{\alpha}, \L^{1}\llcorner_{X_{\alpha}})$; moreover $h(\alpha,\cdot)$ has an almost everywhere 
representative that is continuous on $\overline{X_\alpha}$, and locally Lipschitz and positive in the interior of $X_\alpha$.
\end{itemize}
Moreover, fixed any $\qq$ as above such that $\QQ_{\sharp}( \mm\llcorner_{\T_{V}} ) \ll \qq$, the disintegration is $\qq$-essentially unique in the following sense: if any other 
map $Q \ni \alpha \mapsto \bar \mm_{\alpha} \in \mathcal{P}(X)$
satisfies points (1)-(2), then 
$\bar \mm_{\alpha} = \mm_{\alpha}$ for $\qq$-a.e.  $\alpha \in Q$.
\end{theorem}

To localise curvature bounds, a larger family of Lorentz-Wasserstein geodesics 
was needed: we recall a second way to construct $\ell^{p}$-cyclically monotone sets (introduced in \cite{cava:decomposition} for the metric setting and adapted to the Lorentzian framework in  \cite{CaMo:20}).

\begin{proposition}\label{P:cpgeod}  
Let $\Delta \subset  \Gamma_{V}$ 
be such that, for all  $(x_{0},y_{0}),(x_{1},y_{1}) \in \Delta$:
\begin{equation}\label{E:monotone}
(\tau_{V}(x_{0}) - \tau_{V}(x_{1}))(\tau_{V}(y_{0}) - \tau_{V}(y_{1})) \geq 0.  
\end{equation}
Then $\Delta$ is 
$\ell^{p}$-cyclically monotone for each $p\in (0,1)$.
\end{proposition}


\section{Localization  of Timelike Ricci curvature bounds}
\label{S:localization}

We will improve on the results of \cite{CaMo:20} 
concerning the regularity properties of the marginal measures associated to 
the decomposition induced by $\tau_{V}$.

To obtain the estimates for the one-dimensional densities, it is more convenient to use an equivalent 
form of the $\TCD^{e}_{p}$ condition.  This equivalent form of $\TCD^{e}_{p}$, whose Riemannian counterpart is the well known 
$\CD^{*}(K,N)$ condition of Bacher and Sturm \cite{BS10}, has recently been presented also in 
the Lorentzian setting in \cite{Braun} and is denoted by $\TCD^{*}_{p}$.

We will not use the full equivalence between  $\TCD^{e}_{p}$ and 
$\TCD^{*}_{p}$ proven in \cite{Braun} (see also \cite{EKS} 
for the earlier equivalence between $\CD^{e}$ and $\CD^{*}$), 
but merely that $\TCD^{e}_{p}$ implies $\TCD^{*}_{p}$. For readers' convenience
we now include a self-contained proof of this implication. 

\begin{proposition}\label{P:summaryTCD}
Let  $(X,\sfd, \mm, \ll, \leq, \tau)$ be a 
globally hyperbolic, timelike non-branching Lorentzian geodesic space satisfying 
$\mathsf{TCD}^{e}_{p}(K,N)$  for some $p\in (0,1), K\in \R, N\in [1,\infty)$. 
Let $\mu_{0},\mu_{1} \in \Prob_{c}(X)$ with $\mu_{0},\mu_{1} \in \Dom(\Ent(\cdot|\mm))$ and
assume that there exists $\pi\in \Pi^{p\text{-opt}}_{\leq}(\mu_{0},\mu_{1})$ such that $ \pi( \{\tau>0\})=1$.

Then $\pi$ is the unique element  
of $\Pi^{p\text{-opt}}_{\leq}(\mu_{0},\mu_{1})$ concentrated on $\{\tau > 0\}$.
Accordingly, there exists a unique optimal dynamical plan $\eta$ such that 
$(\ee_{0},\ee_{1})_{\sharp}\eta = \pi$.  
Moreover, the  $\ell_{p}$-geodesic  
$\mu_{t} = (\ee_{t})_{\sharp} \eta$
satisfies $\mu_{t} = \rho_{t} \mm \ll \mm$ for every $t\in [0,1]$, and for $\eta$-a.e. $\gamma \in \TGeo(X)$ it holds that
\begin{equation}\label{E:pointwiseTCD}
\rho_{t}(\gamma_{t})^{-\frac{1}{N}} \geq \sigma^{(1-t)}_{K/N} \left(\tau(\gamma_{0},\gamma_{1})\right) 
\rho_{0}(\gamma_{0})^{-\frac{1}{N}} + \sigma^{(t)}_{K/N} \left(\tau(\gamma_{0},\gamma_{1})\right)  
\rho_{1}(\gamma_{1})^{-\frac{1}{N}},
\end{equation}
for all $t\in [0,1]$.
\end{proposition}

\begin{proof}
The first part of the claim is simply \cref{T:1}.
We are left to prove \eqref{E:pointwiseTCD} 
for the $\ell_{p}$-geodesic induced by the unique optimal dynamical plan $\eta$. 
Fix the map $T$ such that 
$(Id,T)_{\sharp} \mu_{0} = \pi$.
\\Since $\pi(\{\tau > 0 \}) = 1$, there exists  a countable collection of Borel sets 
$A_{n}$ such that 
$$
A_{n}\times T(A_{n}) \subset \{\tau > 0\}, \qquad \mu_{0} \left(\cup_{n \in \N} A_{n}\right) = 1.
$$
Without loss of generality, 
we can assume that 
the sets $\{ A_{n} \}_{n\in \N}$ are pairwise disjoint and $\mu_{0}(A_{n})>0$ 
for each $n \in \N$.
Consider $\pi_{n} : = (Id,T)_{\sharp} \mu_{0}\llcorner_{A_{n}}/\mu_{0}(A_{n})$. Such $\pi_{n}$ is the unique $\ell_{p}$-optimal plan between its marginal measures (that we denote by $\mu_{0,n}$ and $\mu_{1,n}$) and, 
accordingly, $\eta\llcorner_{\ee_{0}^{-1}(A_{n})}/\mu_{0}(A_{n})$ the unique optimal dynamical plan. 
Hence from Definition \ref{def:TCD(KN)} and \eqref{eq:sNconc} it follows that 
$$
U_{N}(\mu_{t,n}|\mm) 
\geq \sigma^{(1-t)}_{K/N} \left(\|\tau\|_{L^{2}(\pi_{n})}\right) U_{N}(\mu_{0,n}|\mm) 
+ \sigma^{(t)}_{K/N} \left( \|\tau\|_{L^{2}(\pi_{n})} \right) U_{N}(\mu_{1,n}|\mm),
$$
for all $t \in [0,1]$ and all $n \in \N$.
From here we can repeat verbatim a classical argument already present in the literature (see \cite[Thm.\;3.12]{EKS}) that yields, by restricting to finer subsets of timelike geodesics 
via $\eta$, the inequality \eqref{E:pointwiseTCD} and therefore the claim.
\end{proof}

\smallskip
\subsection{Localization of timelike Ricci lower bounds to $V$-rays}

The goal of this section is to localize  the timelike Ricci curvature lower bounds $\mathsf{TCD}^{e}_{p}(K,N)$ along the $V$-rays:  
this will be achieved by obtaining the next differential inequality \eqref{eq:DiffIneqCDKN} 
for the densities of the disintegration; in the smooth setting, this procedure corresponds to the Raychaudhuri equations.

\begin{theorem}\label{T:local}
Let  $(X,\sfd, \mm, \ll, \leq, \tau)$ be a timelike non-branching, globally hyperbolic Lorentzian geodesic space 
satisfying $\mathsf{TCD}^{e}_{p}(K,N)$ and assume that the 
causally-reversed structure satisfies the same conditions. 

Let $V\subset X$ be a Borel achronal timelike complete subset
and  consider the disintegration formula given by 
Theorem \ref{T:disint}.

Then, for $\qq$-a.e.\,$\alpha$, the one-dimensional metric measure space 
$(X_{\alpha},|\cdot|, \mm_{\alpha})$ satisfies the classical $\CD(K,N)$; namely,\;writing $\mm_\alpha=h(\alpha, \cdot)\L^{1}\llcorner_{X_{\alpha}}$, then $h(\alpha, \cdot)$ is semi-concave (and thus twice differentiable $\L^{1}$-a.e.  on $X_{\alpha}$) and it satisfies the differential inequality
\begin{equation}\label{eq:DiffIneqCDKN}
\frac{\partial^2}{\partial x^2}\log h(\alpha, x)+\frac{1}{N-1}\left(\frac{\partial}{\partial x}\log h(\alpha, x)\right)^2 \leq -K,
\end{equation}
at any point $x$ in the interior of $X_\alpha$ where $h(\alpha, \cdot)$ is twice differentiable.
\end{theorem}

\begin{proof}
For $x\in \T_{V}$ we will write $R(x)$ to denote its equivalence class in $(\T_{V}, R_{V})$, i.e., the ``ray passing through $x$'' (recall \cref{lem:XalphaI}). For a subset $B\subset \T_{V}$, we denote $R(B):=\bigcup_{x\in B} R(x)$.
\\Let $\bar Q \subset Q$ be an arbitrary compact subset of positive $\qq$-measure for which there
exist $a_{0} < a_{1}$ such that   
\begin{equation*}
\begin{split}
X_{\alpha} \cap \{\tau_{V}= a_{0}\} \neq \emptyset, 
\quad X_{\alpha} \cap \{\tau_{V}= a_{1}\} \neq \emptyset \quad \forall  \alpha \in \bar Q,\\
R(\bar{Q})\cap \tau_{V}^{-1}([a_{0}, a_{1}]) \Subset X,\\
 \{(x,y)\in \Gamma_{V}\,:\, x,y\in R(\bar{Q}), \,  \tau_{V}(x)=a_{0}, \,  \tau_{V}(y)=a_{1} \} \Subset \{\tau>0\}.
\end{split}
\end{equation*}
For any $A_{0},A_{1}\in (a_{0}, a_{1})$  with $A_{0} < A_{1}$,
and $L_{0},L_{1} > 0$ satisfying 
$A_{0} + L_{0} < A_{1} + L_{1} < a_{1}$,  consider the probability measures 
$$
\mu_{0} = \int_{\bar Q} \frac{\mathcal{L}^{1}\llcorner_{X_{\alpha}\cap [A_{0}, A_{0}+L_{0}]}}{L_{0}}\,\qq(d\alpha),
\qquad 
\mu_{1} = \int_{\bar Q} \frac{\mathcal{L}^{1}\llcorner_{X_{\alpha}\cap [A_{1}, A_{1}+L_{1}]}}{L_{1}}\,\qq(d\alpha).
$$
Proposition \ref{P:cpgeod} ensures that the transport plan $\pi$ defined
as the monotone rearrangement along each ray $X_{\alpha}$ of the normalized Lebesgue measure $\mathcal{L}^{1}\llcorner_{X_{\alpha}\cap [A_{0}, A_{0}+L_{0}]}/L_{0}$ to 
$\mathcal{L}^{1}\llcorner_{X_{\alpha}\cap [A_{1}, A_{1}+L_{1}]}/L_{1}$ is 
is $\ell^{p}$-cyclically monotone. Since $\pi(X^{2}_{\ll}) = 1$, we infer that $\pi$ is $\ell_{p}$-optimal thanks to \cref{prop:cicmon<->opt}.

From Theorem \ref{T:disint} it is immediate to observe that $\mu_{0},\mu_{1} \ll \mm$, 
notice indeed that the density $h(\alpha, \cdot)$ is strictly positive in the interior of the transport ray $X_\alpha$, for $\qq$-a.e.\;$\alpha$.
Hence we can invoke Proposition \ref{P:summaryTCD}  to deduce that $\pi$  is the unique element in  $\Pi^{p\text{-opt}}_{\ll}(\mu_{0},\mu_{1})$.  Moreover, there exists a unique optimal dynamical plan $\eta$ such that 
$(\ee_{0},\ee_{1})_{\sharp}\eta = \pi$, and  the  $\ell_{p}$-geodesic  
$\mu_{t} = (\ee_{t})_{\sharp} \eta$
satisfies $\mu_{t} = \rho_{t} \mm \ll \mm$ for every $t\in [0,1]$, and for $\eta$-a.e. $\gamma \in \TGeo(X)$, the concavity estimate \eqref{E:pointwiseTCD} holds.

The $\ell_p$-geodesic $\mu_t$ can be written explicitly. Indeed, consider  
$$
\bar{\mu}_{t} := \int_{\bar Q} \frac{\mathcal{L}^{1}_{X_{\alpha}\cap [A_{t}, A_{t}+L_{t}]}}{L_{t}}\,\qq(d\alpha), 
\quad   A_{t} = A_{0}(1-t) + A_{1} t, \, L_{t} = L_{0}(1-t) + L_{1} t. 
$$
Such $(\bar{\mu}_t)_{t\in [0,1]}$ can be lifted to an optimal dynamical plan $\bar{\eta}$  such that 
$(\ee_{0},\ee_{1})_{\sharp}\bar{\eta} = \pi$. By the uniqueness discussed above, we infer that $\bar{\eta}=\eta$ and thus $\mu_t=\bar{\mu}_t$ for all $t\in [0,1]$.
Since 
$$
\mm\llcorner_{\T_{V}} 
= \int_{Q} \mm_{\alpha}\, \qq(d\alpha)= \int_{Q} h(\alpha,\cdot) \, \L^{1}\llcorner_{X_{\alpha}}\, \qq(d\alpha),
$$
one has that $\mu_{s} = \rho_{s} \mm$, with 
$
\rho_{s}(\alpha,t) = \frac{1}{L_{s} h(\alpha,t)}$,  for all  $t \in [A_{s},A_{s}+L_{s}]$. 
Hence, the concavity estimate \eqref{E:pointwiseTCD} on $\rho_s(\alpha, \gamma_s)$ writes as:
\begin{align*}
(L_{s})^{\frac{1}{N}} h(\alpha, (1-s) t_{0} + s t_{1} )^{\frac{1}{N}}
	\geq 	&	\sigma_{K/N}^{(1-s)}(t_{1}-t_{0}) (L_{0})^{\frac{1}{N}} h(\alpha, t_{0} )^{\frac{1}{N}} \\ & \quad + \sigma_{K/N}^{(s)}(t_{1}-t_{0}) (L_{1})^{\frac{1}{N}} h(\alpha, t_{1} )^{\frac{1}{N}}, 
\end{align*}
for every $s \in [0,1]$, for $\L^{1}$-a.e. $t_{0} \in [A_{0},A_{0} + L_{0}]$ and $t_{1}$ obtained as the image of $t_{0}$ through the monotone rearrangement of $[A_{0},A_{0}+L_{0}]$ to 
$[A_{1},A_{1}+L_{1}]$.
Specializing  the previous inequality for $s = 1/2$ and noticing that  $t_{0} = A_{0} + \tau L_{0}$ gives $t_{1} = A_{1} + \tau L_{1}$, we obtain:
\begin{align*}
&(L_{0} + L_{1})^{\frac{1}{N}} h(\alpha,A_{1/2} + \tau L_{1/2})^{\frac{1}{N}} \\
	&\quad			\geq  
		2^{\frac{1}{N}}\sigma^{(1/2)}_{K/N}( A_{1} - A_{0} + \tau |L_{1} - L_{0}| ) \big\{ (L_{0})^{\frac{1}{N}} h(\alpha,A_{0} + \tau L_{0})^{\frac{1}{N}}\\
        & \qquad \qquad \qquad \qquad \qquad \qquad \qquad \qquad \qquad + (L_{1})^{\frac{1}{N}} h(\alpha, A_{1} + \tau L_{1})^{\frac{1}{N}} \big\},
\end{align*}
for $\L^{1}$-a.e. $\tau \in [0,1]$, where we used the notation $A_{1/2}:=\frac{A_0+A_1}{2}, L_{1/2}:=\frac{L_0+L_1}{2}$. Recalling from Theorem \ref{T:disint} that the map $s \mapsto h(\alpha,s)$ is continuous, we infer that the previous inequality also holds for $\tau =0$:
\begin{equation}\label{E:beforeoptimize}
\begin{split}
&(L_{0} + L_{1})^{\frac{1}{N}} h(\alpha,A_{1/2} )^{\frac{1}{N}}\\
&	\quad 	\geq 
		2^{\frac{1}{N}}\sigma^{(1/2)}_{K/N}( A_{1} - A_{0})
				\left\{ (L_{0})^{\frac{1}{N}} h(\alpha,A_{0})^{\frac{1}{N}} 
				+ (L_{1})^{\frac{1}{N}} h(\alpha,A_{1})^{\frac{1}{N}} \right\},
\end{split}
\end{equation}
for all $A_{0} < A_{1}$  with $A_{0},A_{1}\in (a_0, a_1)$, all sufficiently small $L_{0}, L_{1}$ and $\qq$-a.e. $\alpha \in Q$, 
with exceptional set depending on $A_{0},A_{1},L_{0}$ and $L_{1}$. 

Noticing that all the terms appearing in \eqref{E:beforeoptimize} depend in a continuous way on $A_{0},A_{1},L_{0}$ and $L_{1}$, it follows  from Fubini's Theorem that there 
exists a common exceptional set $N \subset Q$ with the following properties: $\qq(N) = 0$ and for each $\alpha \in Q\setminus N$  the inequality \eqref{E:beforeoptimize} holds true for all  
$A_{0},A_{1},L_{0}$ and $L_{1}$.

Then one can make the following (optimal) choice 
$$
L_{0} : = L \frac{h(\alpha,A_{0})^{\frac{1}{N-1}}  }{h(\alpha,A_{0})^{\frac{1}{N-1}} 
		+ h(\alpha,A_{1})^{\frac{1}{N-1}} }, \qquad 
L_{1} : = L \frac{h(\alpha,A_{1})^{\frac{1}{N-1}}  }{h(\alpha,A_{0})^{\frac{1}{N-1}} 
		+ h(\alpha,A_{1})^{\frac{1}{N-1}} },
$$
for any $L > 0$ sufficiently small, and obtain that 
\begin{equation}\label{E:intermediate}
h(\alpha,A_{1/2} )^{\frac{1}{N-1}}
		\geq 
		2^{\frac{1}{N-1}}\sigma^{(1/2)}_{K/N}( A_{1} - A_{0})^{\frac{N}{N-1}} 
				\left\{  h(\alpha,A_{0})^{\frac{1}{N-1}} + h(\alpha,A_{1})^{\frac{1}{N-1}} \right\}.
\end{equation}
By \cite[Prop.\;5.5]{BS10} we know that for any $K' < \tilde K < K $ there exists $\Theta^{*}> 0$
such that for all $0 \leq \Theta \leq \Theta^{*}$ it holds that $
\sigma_{\tilde K/N}^{(t)}(\theta) \geq \sigma_{K'/(N-1)}^{(t)}(\theta)^{\frac{N-1}{N}} t^{\frac{1}{N}},
$
for all $t \in [0,1]$. Hence 
$$
2^{\frac{1}{N-1}} \sigma_{\tilde K/N}^{(1/2)}(\theta)^{\frac{N}{N-1}} \geq \sigma_{K'/(N-1)}^{(1/2)}(\theta).
$$
Plugging the last inequality into  \eqref{E:intermediate} gives 
$$
h(\alpha,A_{1/2} )^{\frac{1}{N-1}}
		\geq 
		\sigma^{(1/2)}_{K'/N-1}( A_{1} - A_{0})
				\left\{  h(\alpha,A_{0})^{\frac{1}{N-1}} + h(\alpha,A_{1})^{\frac{1}{N-1}} \right\},
$$
for all $A_{0},A_{1}$ sufficiently close. In particular this shows that 
$(X_{\alpha}, |\cdot |, \mm_{\alpha})$ verifies the $\CD_{loc}(K',N)$ condition 
that is easily seen in dimension one to be equivalent to the full $\CD(K',N)$ condition, as they are both equivalent to the differential inequality \eqref{eq:DiffIneqCDKN}, with $K$ replaced by $K'$. 
To prove the claim is then enough to let $K'$ converge to $K$ and invoke the stability of the 
$\CD$ condition \cite{lottvillani, sturm:II} to obtain that $(X_{\alpha}, |\cdot |, \mm_{\alpha})$ verifies $\CD(K,N)$.
\end{proof}

To conclude this part we recall a straightforward consequence of the  $\CD(K,N)$ condition along the $V$-rays:
for all $x_{0},x_{1}\in X_{\alpha}$, 
\begin{equation}\label{E:MCP0N1d}
\left(\frac{\fs_{K/(N-1)}(b-\tau_{V}(x_{1}))}{\fs_{K/(N-1)}(b-\tau_{V}(x_{0}))}\right)^{N-1}
\leq \frac{h(\alpha, x_{1} ) }{h (\alpha, x_{0})} \leq 
\left( \frac{\fs_{K/(N-1)}(\tau_{V}(x_{1}) - a) }{\fs_{K/(N-1)}(\tau_{V}(x_{0}) -a)} \right)^{N-1},  
\end{equation}
with $a< \tau_{V}(x_{0})<\tau_{V}(x_{1})<b$ and $b- a \leq \pi\sqrt{(N-1)/(K\vee 0)}$.
The values $a$ and $b$ correspond to the evaluation of $\tau_V$ at the initial and final point of $X_\alpha$, respectively, whenever they are finite.
In other words, for $\qq$-a.e. $\alpha\in Q$, the one-dimensional metric measure space 
$(X_{\alpha},|\cdot|, \mm_{\alpha})$ 
also satisfies the weaker $\MCP(K,N)$.

\section{Timelike Minkowski content and its properties}
\label{S:isoperimetric}

The goal of the next definition is to provide a notion of ``area" for an achronal set $A\subset X$.
The rough idea is to use the signed time-separation function $\tau_A$ from $A$ to define a  ``future $\epsilon$-tubular neighbourhood" of $A$, and then define the ``area of $A$" as the first variation of the volume of such future-tubular neighbourhoods. This can be seen as a Lorentzian counterpart of the outer Minkowski content in metric measure spaces.

\begin{definition}[Timelike Minkowski content]\label{D:areaachronal}
Let  $A \subset X$ be a Borel achronal set and consider  
the signed time-separation function $\tau_{A}$ from $A$, see \eqref{eq:deftauV}.
We define the \emph{future Minkowski content} of $A$ by 
\begin{equation}\label{E:future}
\mm^{+}(A) : = \inf_{U \in \mathcal{U}} \limsup_{\ve \to 0} \frac{\mm( \tau_{A}^{-1}((0,\ve)) \cap U) }{\ve}, \quad 
\mathcal{U}: = \{ U \subset X \colon  U \text{ open}, \ A\subset U \}.
\end{equation}
We define the \emph{past Minkowski content} of $A$ by 
\begin{equation}\label{E:past}
\mm^{-}(A) : = \inf_{U \in \mathcal{U}}\limsup_{\ve \to 0} \frac{\mm( \tau_{A}^{-1}((-\ve,0))  \cap U)}{\ve}.
\end{equation} 
\end{definition}

The presence of the infimum over the collection $\mathcal{U}$ of open sets containing $A$ is necessary to avoid infinite volume of the future $\ve$-enlargement of $A$ with respect to 
$\tau$; indeed, tipically (e.g. in Minkowski spacetime), $\tau_{A}^{-1}((0,\ve))$ has infinite volume for every $\ve>0$.

\begin{remark}[Timelike Minkowski content equals area in the smooth framework]
In the smooth framework, the timelike Minkowski content 
can be related  to the classical area of $A$, as illustrated below. 
If $(M,g)$ is a globally hyperbolic spacetime and 
$A \subset M$ a smooth, spacelike, acausal and future causally complete hypersurface, then 
the signed time-separation function $\tau_A$ 
is smooth on $I^+(A)$ -- outside of a set of measure zero.
As $\tau_A$ has timelike gradient $\nabla \tau_A$ with $g(\nabla \tau_A,\nabla \tau_A) = -1$, the level sets $\tau_A^{-1}(t)$ are spacelike hypersurfaces of 
$I^+(A)$,  for almost every $t>0$. Denoting by $A_t = \tau_A^{-1}(t)$, 
coarea formula (see for instance \cite[Prop.\;3]{TreudeGrant}) implies that 
$$
\vol_g\llcorner_{I^+(A)} = 
\int_{(0,+\infty)} \vol_{g_t} \, dt, 
$$
where $\vol_{g_t}$ is the volume measure induced by $g_t$, the Riemannian metric induced by $g$ over $A_t$.  
Hence, in the smooth setting above, the timelike Minkowski content coincides with the area induced by the ambient Lorentzian metric $g$.

Concerning the case when $A$ is a smooth null hypersurface, then  
$\mm^+(A)$ has to be zero like the induced volume. Let us briefly sketch the argument. If $x \in I^+(A)$, 
then the supremum defining $\tau_A(x)$ 
cannot be realized as a maximum. Otherwise, by the classical first variation argument, the 
optimal path from $x$ to $A$ has to be a geodesic normal to $A$; but since $A$ is null, the normal directions are contained in the tangent space of $A$. This forces the geodesic to never leave $A$ and yields a contradiction. It follows that all the optimal directions should leave $A$ from its ``boundary" (of higher codimension). Because of the scaling limit in the definition, such a lower dimensional contribution  is not detected by the timelike Minkowski content which thus has to vanish.
\end{remark}

We will use the localization associated to a timelike complete achronal set $V$ to bound from above the 
future and past Minkowski content of an achronal set $A$.
To exclude lightlike variations, we introduce a stronger condition for achronal sets.

\begin{definition}[Empty future $V$-boundary, $\partial_V^+ A = \emptyset$]\label{D:V-boundary}
Let $V\subset X$ be a timelike complete achronal set.
We say that an achronal set $A \subset I^{+}(V)$  has \emph{empty future $V$-boundary}, and we write $\partial_V^+ A = \emptyset$, if the following property is satisfied:
for every $x \in I^{+}(A)$ and every geodesic $\gamma : [0,1] \to X$ with $\gamma_{0} \in V$, $\gamma_{1} = x$,   such that $\tau_{V}(x) = \tau(\gamma_{0},x)$, 
it holds that $\gamma_{[0,1]} \cap A \neq \emptyset$.

If $A$ satisfies the reversed condition in the past we say that $A$ has \emph{empty past $V$-boundary} and we write $\partial_V^- A = \emptyset$. In case $A$ has both past and future $V$-boundaries empty, then we say that $A$ has empty $V$-boundary and we write $\partial_V A= \emptyset$.
\end{definition}

It is natural to compare \cref{D:V-boundary} with the fundamental notion of Cauchy hypersurface. 
Recall that a 
Cauchy hypersurface is a closed achronal set intersected exactly once by any inextendible causal curve \cite[Sect.\;8.3]{Wald}.
In case $X$ is a smooth manifold with a continuous Lorentzian metric, then a Cauchy hypersurface is a closed  achronal topological hypersurface \cite[Prop.\;5.2]{SaC0}. 

\begin{lemma}\label{L:Cauchyboundary}
Let  $V\subset X$ be a Borel achronal timelike complete subset, and let  $A\subset X$ be a Cauchy hypersurface. Then $A$ has empty $V$-boundary.
\end{lemma}

\begin{proof}
Let $\gamma$ be a geodesic as in \cref{D:V-boundary} and let $\bar \gamma$ to be any maximal causal extension of $\gamma$. 
Then, by definition, $\bar \gamma$ has to meet $A$ 
and by the reverse triangle inequality
$\gamma$ has to meet $A$. 
\end{proof}

\begin{remark}
An elementary example (taking, for instance, a $Y$ shaped spacetime) shows that the reverse implication of \cref{L:Cauchyboundary} fails to be valid. 
\end{remark}
For the dimensional reduction we need also to specify the following notation. Let $(I,|\cdot|, \nu)$ be a one-dimensional metric measure space  
with $I$ a closed interval and $\nu$ a non-negative Radon measure.
For a Borel set  $A \subset  I$, we denote
\begin{equation}\label{E:onedright}
\nu^{+}(A) := \limsup_{\ve \to 0} \frac{\nu(\cup_{x\in A}(x,x+\ve))}{\ve}, \qquad 
\nu^{-}(A) := \limsup_{\ve \to 0} \frac{\nu(\cup_{x\in A}(x-\ve,x))}{\ve}.
\end{equation}
We will say that $\nu^{+}(A)$ (resp.\;$\nu^{-}(A)$) is the future (resp.\;past) Minkowski content of $A$.
Note that if $A$ is bounded and $\nu = f(x)\,dx$ with $f$ continuous, then $\limsup$ in \eqref{E:onedright} is actually a limit;  moreover, if $A=[a,b]$ is a bounded interval (closeness is not important here), then $\nu^+(A)=f(b)$.
\subsection{Properties of timelike Minkowski contents}

\subsubsection{The unbounded case}\label{SS:UnboundedA}
In this section, we study the timelike Minkowski content of a \emph{possibly unbounded} Borel achronal set.

\begin{proposition}\label{P:main1}
Let $(X,\sfd, \mm, \ll, \leq, \tau)$ be a timelike nonbranching, globally hyperbolic Lorentzian geodesic space satisfying $\mathsf{TCD}^{e}_{p}(K,N)$ and assume that the causally reversed structure satisfies the same conditions. 
Let $V\subset X$ be a complete Borel achronal timelike complete subset and consider the disintegration given by
Theorem \ref{T:disint}. Then the following hold: 
\begin{itemize}
\item For any Borel achronal set $A \subset I^{+}(V)$ with  $\inf_{x\in A} \tau_V(x)>0$ and $\partial_V^+A = \emptyset$, it holds that 
\begin{equation}\label{E:inequality}
\mm^{+}(A) \leq \int_{Q} \mm_{\alpha}^{+}(A\cap X_{\alpha}) \,\qq(d\alpha), 
\end{equation}
where we adopt the notation \eqref{E:onedright} for $\mm_{\alpha}^{+}(A\cap X_{\alpha})$.

If $A \subset I^{-}(V)$ is a Borel, achronal set with $\inf_{x\in A} -\tau_V(x)>0$ and $\partial_V^-A = \emptyset$, then \eqref{E:inequality} holds replacing $\mm^{+}(A)$ by 
$\mm^{-}(A)$ and $\mm^{+}_{\alpha}$ by $\mm^{-}_{\alpha}$.

\item The following inequality holds true
\begin{equation}\label{E:goq}
\mm^{+}(V)  \geq \int_{Q} \mm_{\alpha}^{+}(V\cap \overline{X_{\alpha}}) \,\qq(d\alpha).
\end{equation}
The inequality \eqref{E:goq} remains valid if we replace $\mm^{+}(V)$ by 
$\mm^{-}(V)$ and $\mm^{+}_{\alpha}$ by $\mm^{-}_{\alpha}$ (recall \eqref{E:onedright}).
\end{itemize}
\end{proposition}

\begin{proof}
{\bf Step 1.}
We will prove the first claim for $A \subset I^{+}(V)$. The one for $A \subset I^{-}(V)$ follows by reversing the causal structure.

From $A \subset I^+(V)$, it follows that $I^+(A) \subset I^+(V)$ 
implying the inclusion $\tau_{A}^{-1}((0,\ve)) \subset I^{+}(V)$.
Therefore applying \cref{lem:I+VTV} and  \cref{T:disint}  to $V$, we get
$$
\mm(\tau_{A}^{-1}((0,\ve))) = \int_{Q} \mm_{\alpha}(\tau_{A}^{-1}((0,\ve)) \cap X_{\alpha}) \,\qq(d\alpha).
$$
Consider $X_{\alpha}$ such that $\tau_{A}^{-1}((0,\ve)) \cap X_{\alpha} \neq \emptyset$. By definition, it holds that
\begin{align*}
\tau_{A}^{-1}((0,\ve)) \cap X_{\alpha} = &~ \{ y \in X_{\alpha} \cap I^{+}(A) \mid  \tau (x,y) < \ve, \forall x \in A \}.
\end{align*}
Since $\partial_V^+A = \emptyset$, necessarily 
 $A \cap \overline{X_{\alpha}} \neq \emptyset$; in fact, 
 $A \cap X_{\alpha} \neq \emptyset$, since $A\subset I^+(V)$ and $X_\alpha$ is a $V$-ray.
As  $A$ is achronal 
there cannot be two distinct points in $A \cap X_{\alpha}$. Therefore $A \cap X_{\alpha} = \{ a_{\alpha}\}$
and
\begin{align*}
\tau_{A}^{-1}((0,\ve)) \cap X_{\alpha} \subset &~ \{ y \in X_{\alpha} \cap I^{+}(A) \mid  \tau (a_{\alpha},y) < \ve \} \\
= &~
(A \cap X_{\alpha})^{\ve}\cap X_{\alpha},
\end{align*}
where with $(A \cap X_{\alpha})^{\ve}\cap X_{\alpha}$ we denote the right $\ve$-enlargement of the set $A\cap 
X_{\alpha}$
in the metric measure space 
$(\overline{X_{\alpha}},|\cdot|,\mm_{\alpha})$,  i.e.,\;in the sense of \eqref{E:onedright}. Hence,
\begin{equation}\label{eq:mtauA}
\frac{\mm(\tau_{A}^{-1}((0,\ve))) }{\ve} 
\leq  \int_{Q} \frac{\mm_{\alpha} ((A\cap X_{\alpha})^{\ve})}{\ve} \,\qq(d\alpha), \quad \forall \ve>0.
\end{equation}
Since $A \cap X_{\alpha}  = \{ a_{\alpha}\}$, then $\mm_{\alpha}(A) = \mm_{\alpha}(A \cap X_{\alpha}) = 0$. 
To conclude the argument, we wish to use Fatou's Lemma in order to pass to the limit in the right hand side of \eqref{eq:mtauA}. To this aim, we look for a function  $g \in L^{1}(Q,\qq)$ such that
$$
\frac{\mm_{\alpha} ((A\cap \overline{X_{\alpha}})^{\ve}) }{\ve} \leq g(\alpha), \quad \text{for $\qq$-a.e. $\alpha\in Q$}.
$$
By applying \eqref{E:MCP0N1d} (taking as initial point $0$ that can be identified with $X_{\alpha} \cap V$) we obtain for all $\ve\in (0, \ve_0(K,N))$:
\begin{align*}
\frac{\mm_{\alpha} ((A\cap X_{\alpha})^{\ve}) }{\ve} 
&~ = \frac{1}{\ve}\int_{(a_{\alpha},a_{\alpha}+\ve)}  h(\alpha,s)\,ds \\ 
&~  \overset{\eqref{E:MCP0N1d}}{\leq} \frac{h(\alpha,a_{\alpha})}{\ve}\int_{(a_{\alpha},a_{\alpha}+\ve)} \left( \frac{\fs_{K/(N-1)}(s) }{\fs_{K/(N-1)}(a_\alpha)} \right)^{N-1}  \,ds \\
&~ \leq  C(K,N,\inf_A\tau_V)\; h(\alpha,a_{\alpha}) \\
&~ =  C(K,N,\inf_A\tau_V)\; \mm_{\alpha}^{+}(A\cap X_{\alpha}).
\end{align*}
If $Q \ni \alpha \mapsto \mm_{\alpha}^{+}(A\cap X_{\alpha})$ is $\qq$-integrable, then we can choose  $g(\alpha):=  C(K,N,\inf_A\tau_V)\;\mm_{\alpha}^{+}(A\cap X_{\alpha})$ as majorant and use Fatou's Lemma to pass to the limit in \eqref{eq:mtauA} and obtain 
$$
\mm^{+}(A) \leq \int_{Q} \mm_{\alpha}^{+}(A\cap X_{\alpha})\,\qq(d\alpha).
$$
If $\mm_{\alpha}^{+}(A\cap X_{\alpha})$ is not $\qq$-integrable then the claim holds trivially.

\smallskip
{\bf Step 2.}
We now turn to the second claim.
Fix any ray $X_\alpha$ from the disintegration associated to $V$. 
Since $V$ is achronal, 
there exists a unique $x_{\alpha}$ such that $V \cap \overline{X_{\alpha}} = \{ x_{\alpha}\}$ 
and therefore  for any $y \in X_{\alpha}$  
it holds that $\tau(x_{\alpha}, y ) = \tau_{V}(y)=  \sup_{x \in V} \tau(x,y)$.
Hence:
\begin{align*}
\tau_{V}^{-1}((0,\ve)) \cap X_{\alpha} = &~ \{ y \in X_{\alpha} \mid  0 < \tau_{V} (y) < \ve\}  \\
= &~
\{ y \in X_{\alpha} \mid  0 < \tau (x_{\alpha},y) < \ve \} \\
= &~
(V \cap \overline{X_{\alpha}})^{\ve}\cap X_{\alpha},
\end{align*}
where by
$(V \cap \overline{X_{\alpha}})^{\ve}\cap X_{\alpha}$ we denote the right $\ve$-enlargement of the set $V\cap \overline{X_{\alpha}}$
in $X_{\alpha}$, see \eqref{E:onedright}.
If $U$ is any open set containing $V$, then 
$\tau_{V}^{-1}((0,\ve)) \cap X_{\alpha}\cap U = (V \cap \overline{X_{\alpha}})^{\ve}\cap X_{\alpha} \cap U$.

Hence (recall that $\mm(V) = 0$): 
$$
\frac{\mm(\tau_{V}^{-1}((0,\ve)) \cap U)}{\ve} 
=  \int_{Q} \frac{\mm_{\alpha} ((V\cap \overline{X_{\alpha}})^{\ve} \cap U) }{\ve} \,\qq(d\alpha).
$$
Using  Fatou's Lemma we deduce that
$$
\liminf_{\ve \to 0} \frac{\mm(\tau_{V}^{-1}((0,\ve) \cap U)}{\ve} \geq  
\int_{Q} \mm_{\alpha}^{+}(V\cap \overline{X_{\alpha}}) \,\qq(d\alpha),
$$
notice indeed that the dependence on $U$ on the right hand side disappears after the liminf. 
Finally, taking the infimum over all open sets $U$ containing $V$, we obtain the claim.
\end{proof}

\begin{remark}\label{R:opensets}
\begin{itemize}
\item Notice that in the proof of the first claim of \cref{P:main1} we have also shown that, for achronal sets with empty future $V$-boundary, the restriction to open sets 
present in the definition of future Minkowski content \eqref{E:future} is not necessary to obtain a finite quantity,  provided $\mm_{\alpha}^{+}(A\cap \overline{X_\alpha})$ is $\qq$-integrable.
\item It is clear from the proof that, in the first claim in Proposition \ref{P:main1}, instead of $\inf_A \tau_V>0$ (resp.\;$\inf_A -\tau_V>0$) it is sufficient to assume that $$\qq{\rm\text{-}ess}\inf\{\tau_V(X_\alpha\cap A)\colon \alpha\in Q \text{ s.t. } X_\alpha\cap A\neq \emptyset \}>0,
$$
or, respectively, $\qq{\rm\text{-}ess}\inf\{-\tau_V(X_\alpha\cap A)\colon \alpha\in Q \text{ s.t. } X_\alpha\cap A\neq \emptyset \}>0$.
\end{itemize}
\end{remark}

In the next section we will obtain a monotonicity formula for the rescaled area of the spacelike hypersurface $V_{t} := \{\tau_{V} = t\}$.
As not all the integral lines of $\tau_{V}$ will be longer than $t$, it will be enough to consider the following subset of $Q$: 
\begin{align}
Q_{t} &: = \{ \alpha \in Q \colon \sup_{x \in X_{\alpha}} \tau_{V}(x) > t \}, \quad  t>0, \label{E:rayslonger+}\\
Q_{t} &: = \{ \alpha \in Q \colon \inf_{x \in X_{\alpha}} \tau_{V}(x) < t \},   \quad t<0. \label{E:rayslonger-}
\end{align}
Notice that the equidistant set $\{\tau_{V} = t\}$ can be obtained as the translation at (signed) distance $t$ along the $V$-rays $X_{\alpha}$, for $\alpha \in Q_{t}$. 

\begin{proposition}\label{P:identityslice}
Let  $(X,\sfd, \mm, \ll, \leq, \tau)$  and $V \subset X$ satisfy the same assumption of \cref{P:main1}.
Then 
\begin{equation}\label{E:identityslicesV}
\mm^{+}(V_{t}) = \int_{Q_{t}} \mm^{+}_{\alpha}(V_{t}\cap X_{\alpha}) \,\qq(d\alpha),\qquad  
\mm^{-}(V_{t})  = \int_{Q_{t}} \mm^{-}_{\alpha}(V_{t}\cap X_{\alpha}) \,\qq(d\alpha),
\end{equation}
for $t >0$ and $t < 0$, respectively.
\end{proposition}

\begin{proof}
By symmetry, we will only deal with the case $t >0$.

First, we verify that $\partial_V^+ V_{t} = \emptyset$.
If $z\in I^{+}(V_{t})$ then $\tau_{V}(z) > t$. By continuity of $\tau_{V}$, any maximizing geodesic realizing $\tau_{V}$ going from $z$ to $V$ has to meet $V_{t}$.

Next, we claim that
\begin{equation}\label{eq:intQ=intQt}
\int_{Q} \mm^{+}_{\alpha}(V_{t}\cap X_{\alpha}) \,\qq(d\alpha)=\int_{Q_t} \mm^{+}_{\alpha}(V_{t}\cap X_{\alpha}) \,\qq(d\alpha).
\end{equation}
Indeed,  thanks to the hypothesis $t \neq 0$,
the only $V$-rays contributing in the integral in the left hand side of \eqref{eq:intQ=intQt} are those $X_{\alpha}$ for which 
$X_{\alpha} \cap V_{t} \neq \emptyset$. Since the conditional measures $\mm_\alpha$ are absolutely continuous with respect to the Lebesgue measure $\L^1$ on $X_\alpha$ (see \cref{T:disint}), 
it is enough to take the integral in the left hand side  over $Q_{t}$, i.e., on those rays strictly longer than $t$ (otherwise the right Minkowski content on the ray would be 0). This proves \eqref{eq:intQ=intQt}.
Then the inequality $\leq$ in  the first identity of \eqref{E:identityslicesV} follows from \eqref{E:inequality} combined with \eqref{eq:intQ=intQt}.  

\smallskip
We next show the reverse inequality $\geq$ in  the first identity of \eqref{E:identityslicesV}.
First, we claim that $\tau_{V_t}$, the time-separation function from $V_{t}$, satisfies: 
\begin{equation}\label{eq:tauV=tauVt+t}
\tau_{V_t}(w)=\tau_{V}(w)-t, \quad \forall w\in  I^+(V_t).
\end{equation}

Let  $w \in I^{+}(V_{t}) \subset I^{+}(V)$. By \cref{lem:I+VTV} there exist  $z_{t} \in V_{t}$ and $\alpha \in Q$ such that $z_{t} \in X_{\alpha}$ and 
$w \in \overline{X_{\alpha}}$; denoting by $z$ the unique element of the set $\overline{X_{\alpha}}\cap V$, we obtain
$$
\tau_{V_{t}}(w) \geq \tau(z_{t},w) = \tau (z,w) - \tau(z,z_{t}) = \tau (z,w)  - t = \tau_{V}(w) - t.
$$
To show that $\tau_{V_{t}}(w) \leq \tau_{V}(w) - t$ we argue as follows: for each $\ve > 0$ there exists $\zeta_{t} \in V_{t}$ such that 
$\tau_{V_{t}}(w) \leq \tau(\zeta_{t},w) + \ve$. Since $V$ is timelike complete, by \cref{L:initialpoint}   there exists $\zeta \in V$ such that 
$t = \tau_V(\zeta_t) = \tau(\zeta,\zeta_t)$.
The reverse triangle inequality $\tau(\zeta, w) \geq \tau(\zeta, \zeta_{t}) + \tau(\zeta_{t}, w)$ implies that
$$
\tau_{V_{t}}(w) \leq \tau(\zeta_{t},w) + \ve \leq \tau(\zeta,w) - \tau(\zeta, \zeta_{t}) + \ve  =
\tau(\zeta,w) - t + \ve  \leq  \tau_{V}(w) - t + \ve.
$$
Since $\ve>0$ was arbitrary, we conclude that $\tau_{V_{t}}(w)\leq  \tau_{V}(w) - t$. This completes the proof of  \eqref{eq:tauV=tauVt+t}.

From \eqref{eq:tauV=tauVt+t} and the disintegration \eqref{E:disintegration}, it  follows that for every open set $U\supset V_t$ it holds that
\begin{align}
\mm \left(\tau_{V_t}^{-1}((0,\ve))\cap U \right)&= \mm\left(\tau_V^{-1}((t, t+\ve))\cap U \right) \nonumber \\
&=\int_{Q_{t}} \mm_{\alpha} (\tau_{V}^{-1}((t,t+\ve))  \cap X_{\alpha} \cap U) \,\qq(d\alpha). \label{eq:mVtDisint}
\end{align}
Reasoning as in the second part of the proof of Proposition \ref{P:main1} using Fatou's Lemma, we obtain
\begin{equation}\label{eq:limitRHS}
\liminf_{\ve\to 0^+} \frac{1}{\ve} \int_{Q_{t}} \mm_{\alpha} (\tau_{V}^{-1}((t,t+\ve))  \cap X_{\alpha} \cap U ) \,\qq(d\alpha)\geq \int_{Q_{t}} \mm^{+}_{\alpha}(V_{t}\cap X_{\alpha}) \,\qq(d\alpha).
\end{equation}
To obtain the inequality $\geq$ in the first of \eqref{E:identityslicesV}
it is then enough to take the $\limsup$ as $\ve \to 0^+$ in \eqref{eq:mVtDisint} divided by $\ve>0$, use \eqref{eq:limitRHS}   and finally take the infimum over all open sets $U\supset V_t$.
\end{proof}

\subsubsection{The bounded case}\label{SSec:MinkBounded}

The assumption on the geometry of the boundary of $A$, used in \cref{SS:UnboundedA}, can be replaced by two other assumptions:
that the closure of $A$ is acausal (stronger than achronality assumed in \cref{SS:UnboundedA}) and the existence of an open set $U_0$ such that 
\begin{itemize}
\item $\overline A \subset U_0$;
\item $\mm\big(J^-(U_0) \cap J^+(V)\big)<\infty$,  where $V\subset X$ is a Borel achronal timelike 
complete subset (e.g., a Cauchy hypersurface) such that $A\subset I^+(V)$.
\end{itemize}
In Corollary \ref{C:ineqCompactAcausal}, we will show that compactness and acausality of $A$ will suffice, up to a minor technical assumption  (satisfied if the ambient spacetime $X$ has empty null-or-spacelike boundary or, more generally, if at least $A$ does not intersect it).

\begin{proposition}\label{P:main1_bounded}
Let $(X,\sfd, \mm, \ll, \leq, \tau)$ be a timelike nonbranching, globally hyperbolic Lorentzian geodesic space satisfying $\mathsf{TCD}^{e}_{p}(K,N)$ and assume that the causally reversed structure satisfies the same conditions. 

Let $V\subset X$ be a Borel achronal timelike complete subset
and consider the disintegration given by
Theorem \ref{T:disint}. Then the following hold: 
\begin{itemize}
\item If $A \subset I^{+}(V)$ is a Borel set such that  $\overline A$ is acausal, $\inf_{x\in A} \tau_V(x)>0$, 
and there exists an open set $U_0\subset X$ such that 
\footnote{As it will be clear from the proof, in the case $K\geq 0$, the assumption that  $\sup_{x\in U_0} \tau_V(x)<\infty$ can be dropped.}
\begin{equation}\label{eq:HpU0Prop49}
\overline A \subset U_0, \quad \mm(J^-(U_0)\cap J^+(V))<\infty  
\quad \text{and}\quad  \sup_{x\in U_0} \tau_V(x)<\infty, 
\end{equation}
then
\begin{equation}\label{E:inequalitygeneral}
\mm^{+}(A) \leq \int_{Q} \mm_{\alpha}^{+}(\overline A\cap X_{\alpha}) \,\qq(d\alpha), 
\end{equation}
where we adopt the notation \eqref{E:onedright} for $\mm_{\alpha}^{+}(A\cap X_{\alpha})$.

\item
If $A \subset I^{-}(V)$ is a Borel set 
such that $\overline A$ is acausal,
$\sup_{x\in A} \tau_V(x)<0$, and there exists an open set $U_0\subset X$ such that 
\begin{equation}\label{eq:HpU0Prop49Past}
\overline A \subset U_0, \quad \mm(J^+(U_0)\cap J^-(V))<\infty \quad \text{and}\quad  \inf_{x\in U_0} \tau_V(x)>-\infty,
\end{equation}  then \eqref{E:inequalitygeneral} holds replacing $\mm^{+}(A)$ by 
$\mm^{-}(A)$ and $\mm^{+}_{\alpha}$ by $\mm^{-}_{\alpha}$.
\end{itemize}
\end{proposition}

\begin{proof}
We prove the first statement, for $A \subset I^{+}(V)$, the proof of the second statement, for $A \subset I^{-}(V)$, is completely analogous.

{\bf Step 1.} \\
From $A \subset I^+(V)$, it follows that 
$I^+(A) \subset I^+(V)$ 
implying the inclusion 
$\tau_{A}^{-1}((0,\ve)) \subset I^{+}(V)$.
Therefore applying \cref{lem:I+VTV} and 
\cref{T:disint} to $V$, we get
\begin{equation}\label{E:disintproofalter}
\mm(\tau_{A}^{-1}((0,\ve)) \cap U) = \int_{Q} \mm_{\alpha}(\tau_{A}^{-1}((0,\ve)) \cap X_{\alpha} \cap U) \,\qq(d\alpha),
\end{equation}
where $U$ is any open set containing $A$ 
and the $X_\alpha$'s are rays with respect to $\tau_V$.
The previous integral can be restricted to 
$$
Q_A 
: = \{ \alpha \in Q \colon X_\alpha \cap I^+(A) \neq \emptyset \}.
$$
For each $\alpha \in Q_A$, the 
infimum of the set $X_\alpha \cap I^+(A)$ with respect to 
the $\ll$ relation is a single point, 
that we denote by  $z_\alpha$; 
to see this, recall that $X_\alpha$ is a timelike geodesic 
and that the $\ll$ relation restricted 
to $X_\alpha\subset X$ is equivalent to the standard order
on $\R$ resticted to the interval 
parametrizing $X_\alpha$ by $\tau$-arclength. 
It is easily checked that  $z_\alpha \in X_\alpha$ and $z_\alpha \in \overline{I^+(A)}\setminus I^+(A)$, 
yielding that $\tau_A(z_\alpha) = 0$. 

With a slight abuse of notation, for every $\ve\in \big(0, \tau(z_\alpha, \mathfrak{b}_\alpha)\big)$, we denote by $z^\ve_\alpha$ the unique element of $X_\alpha$ such that 
$$
\tau(z_\alpha,z^\ve_\alpha ) = \ve. 
$$
For  $\ve\geq \tau(z_\alpha,\mathfrak{b}_\alpha)$, i.e., when $X_\alpha$ is not long enough, 
 we set $z^\ve_\alpha = \mathfrak{b}_\alpha$, the end-point of the $V$-ray $X_\alpha$.

Since $z^\ve_\alpha \in I^+(A)$, 
for any $x \in A$
$$
\tau(x,z^\ve_\alpha) \geq 
\tau(x, z_\alpha) + \tau(z_\alpha, z^\ve_\alpha) \geq \ve,   
$$
yielding that $\tau_A(z^\ve_\alpha) \geq \ve$. Therefore,  for all 
$\alpha \in Q_A$,
$$
\tau_{A}^{-1}((0,\ve)) \cap X_{\alpha}
\subset \{ z \in X_\alpha \colon  
z_\alpha \ll z \ll z^\ve_\alpha\}. 
$$
For ease of notation, we will denote 
$\{ z \in X_\alpha \colon  
z_\alpha \ll z \ll z^\ve_\alpha\}$ by 
$(z_\alpha, z^\ve_\alpha)$. 
We continue from \eqref{E:disintproofalter},
obtaining   
\begin{align*}
\frac{\mm(\tau_{A}^{-1}((0,\ve)) \cap U)}{\ve} 
&~\leq  \int_{Q_A} \frac{\mm_{\alpha}( (z_\alpha, z^\ve_\alpha) \cap U)}{\ve} \,\qq(d\alpha) \\
&~=
\int_{Q_{A,U}} \frac{\mm_{\alpha}( (z_\alpha, z^\ve_\alpha) \cap U)}{\ve} \,\qq(d\alpha),
\end{align*}
where now 
$Q_{A,U} = 
\{\alpha \in Q_A \colon [z_\alpha, \mathfrak{b}_\alpha] \cap U \neq \emptyset\}$.

To conclude the argument, we wish to use Fatou's Lemma in order to pass to the limit in the right hand side of the previous integral. To this aim, we look for a function  $g \in L^{1}(Q,\qq)$ such that
$$
\frac{\mm_{\alpha} ((z_\alpha, z^\ve_\alpha) \cap U)}{\ve} \leq g(\alpha), \quad \text{for $\qq$-a.e. $\alpha\in Q_{A,U}$}.
$$
By applying \eqref{E:MCP0N1d}, taking as initial point $0$ that can be identified with $X_{\alpha} \cap V$, we obtain for all $\ve\in (0, \ve_0(K,N))$:
\begin{align*}
\frac{\mm_{\alpha} ((z_\alpha, z^\ve_\alpha) \cap U) }{\ve} 
&~ \leq \frac{1}{\ve}\int_{(t_{\alpha},t_{\alpha}+\ve)}  h(\alpha,s)\,ds \\ 
&~  \overset{\eqref{E:MCP0N1d}}{\leq} \frac{h(\alpha,t_{\alpha})}{\ve}\int_{(t_{\alpha},t_{\alpha}+\ve)} \left( \frac{\fs_{K/(N-1)}(s) }{\fs_{K/(N-1)}(t_\alpha)} \right)^{N-1}  \,ds \\
&~ \leq  C(K,N,\inf_A\tau_V)\; h(\alpha,t_{\alpha}),
\end{align*}
where $t_\alpha=\tau_V(z_\alpha)$. 

We will now specialize this construction to $U_0$ and prove the integrability of $\alpha \mapsto h(\alpha,t_\alpha)$ over $Q_{A,U_0}$. 
Notice that so far no regularity assumptions on $A$ nor on $U$ where used.

\medskip
{\bf Step 2.} 
 Thanks to \cref{T:local}, each $h(\alpha, \cdot)$ defines a $\CD(K,N)$ density on its domain. Using standard one-dimensional estimates on $\CD(K,N)$ densities on the real line (see for instance \cite[Lemma A.8]{CMi}), we infer that
 $$
h(\alpha,t_\alpha) \leq \frac{F_{K,N}(t_\alpha)}{t_\alpha} \int_{(0,t_\alpha)} h(\alpha,s) \,ds,
$$
for each $\alpha \in Q_{A,U_0}$, where:
\begin{itemize}
\item if  $K\geq 0$, then 
$
F_{K,N}(\cdot)
\equiv N$, 
\item if  $K\leq 0$, then $F_{K,N}:[0,\infty)\to [0,\infty)$ is continuous and $$\lim_{t\to \infty} F_{K,N}(t)=\infty.$$
\end{itemize}

From the assumption that $\inf_{x\in A} \tau_V(x) = c >0$, we deduce the following: since $z_\alpha \in J^+(A)$, there exists $x \in A$ such that $x \leq z_\alpha$ and another $w \in V$ such that $w \leq x$ and $\tau_V (x) = \tau(w,x)$; then 
$$
t_\alpha = \tau_V(z_\alpha) \geq \tau(w,z_\alpha) \geq 
\tau(w,x) + \tau(x,z_\alpha) \geq \tau(w,x) = \tau_V(x) \geq c.
$$
Hence 
for each $\alpha \in Q_{A,U_0}$
\begin{equation}\label{eq:suphat}
h(\alpha,t_\alpha) \leq \frac{F_{K,N}(t_\alpha)}{c} \int_{(0,t_\alpha)} h(\alpha,s) \,ds.
\end{equation}
Then, since 
$$
\bigcup_{\alpha \in Q_{A,U_0}} \big\{ x \in X_\alpha \colon 0 < \tau_V(x) < t_\alpha \big\} \subset I^{-}(U_0) \cap J^+(V), 
$$
the finiteness of $\mm(I^-(U_0) \cap J^+(V))$ ensures that
$$
\int_{Q_{A,U_0}}  \left(  \int_{(0, t_{\alpha})} h(\alpha, s) \, ds \right) \qq(d\alpha) \leq \mm(I^-(U_0) \cap J^+(V))<\infty,  
$$
yielding that the function 
\begin{equation}\label{eq:IntegrableOverQAUU}
Q_{A,U_0}\ni \alpha \mapsto \int_{(0, t_{\alpha})} h(\alpha, s) \, ds \quad \text{is $\qq$-integrable}.
\end{equation}
Recalling that, by assumption,  $\sup_{x\in U_0} \tau_V(x)<\infty$, we also infer that
\begin{equation}\label{eq:FKNtal<inf}
\sup_{\alpha \in Q_{A,U_0}} F_{K,N}(t_\alpha)<\infty.
\end{equation}
Combining \eqref{eq:suphat}, \eqref{eq:IntegrableOverQAUU} and \eqref{eq:FKNtal<inf}, we obtain
that 
\begin{equation}\label{eq:SupIntegrableOverQAUU}
Q_{A,U_0}\ni \alpha \mapsto \sup_{t\in (0, t_{\alpha})} h(\alpha, t)  \quad \text{is $\qq$-integrable}.
\end{equation}
Hence, for each open set $U$ such that   
$A\subset U \subset U_0$,
we can repeat the same argument (thanks to 
the fact that also $I^-(U) \cap J^+(V)$ 
has finite $\mm$-measure) and obtain:
$$
\mm^{+}(A) \leq \int_{Q_{A,U}} \mm_{\alpha}^{+}(\{z_\alpha\})\,\qq(d\alpha),
$$
where
$$
Q_{A,U} : = \{\alpha \in Q_A \colon [z_\alpha, \mathfrak{b}_\alpha] \cap U \neq \emptyset \}.
$$

\smallskip
{\bf Step 3.} 
Consider now a decreasing sequence of open sets 
$U_n$ such that $\bigcap_n U_n = \overline{A}$ and $\overline U_{n+1} \subset U_n$ (so that also $\bigcap_n \overline U_n = \overline{A}$).
Taking the limit as $n \to \infty$, gives 
\begin{equation}\label{E:almosthesis}
\mm^{+}(A) \leq 
\int_{\cap_n Q_{A,U_n}} \mm_{\alpha}^{+}(\{z_\alpha\})\,\qq(d\alpha). 
\end{equation}
It remains to prove that:
$$
\cap_n Q_{A,U_n} = 
\{\alpha \in Q_A \colon [z_\alpha, \mathfrak{b}_\alpha] \cap U_n \neq \emptyset , \text{for all } n \in \N\} \subset 
\{\alpha \in Q_A \colon z_\alpha \in \overline A \}.
$$
Let $\alpha \in \cap_n Q_{A,U_n}$. For each $n\in \N$
$$
\emptyset \neq ([z_\alpha,\mathfrak{b}_\alpha]
\cap U_n )\subset 
([z_\alpha,\mathfrak{b}_\alpha]
\cap \overline{U_n} ).
$$
Since $[z_\alpha,\mathfrak{b}_\alpha]
\cap \overline{U_n}$ forms a decreasing sequence of closed sets and the space is complete, we infer that
$$
\emptyset \neq \bigcap_{n\in\N} [z_\alpha,\mathfrak{b}_\alpha]
\cap \overline{U_n}
= [z_\alpha,\mathfrak{b}_\alpha]
\bigcap_{n\in\N} 
\overline{U_n} = 
[z_\alpha,\mathfrak{b}_\alpha] \cap \overline A.
$$
Finally, if  $w_\alpha \in ([z_\alpha,\mathfrak{b}_\alpha] \cap \overline A)$ 
and $w_\alpha \neq z_\alpha$,  we reach a contradiction:
since $z_\alpha\ \in J^+(A)$, 
there will exist $x_\alpha \in A$
such that $x_\alpha \leq z_\alpha \ll w_\alpha$ implying a violation of the acausality of $\overline A$. 
Hence $\cap_n Q_{A,U_n}\subset  \{\alpha \in Q \colon z_\alpha \in \overline A \}$ and 
from \eqref{E:almosthesis} we obtain 
$$
\mm^+(A) \leq 
\int_{\{\alpha \in Q_A \colon z_\alpha \in \overline A\}}
\mm_{\alpha}^+(\{z_\alpha \})\,\qq(d\alpha).
$$
By acausality of $A$, if $z_\alpha \in \overline A$ then
$\{ z_\alpha \} = \overline A \cap X_\alpha$. 
Therefore 
$$
\{\alpha \in Q_A \colon z_\alpha \in \overline A\} \subset 
\{\alpha \in Q_A \colon 
\overline A \cap X_\alpha \neq \emptyset\}.
$$
Combining the considerations above, we conclude that
$$
\mm^+(A) \leq 
\int_{Q}
\mm_{\alpha}^+(\overline A \cap X_\alpha)\,\qq(d\alpha).
$$
\end{proof}

If $A$ is assumed to be compact, the statement of \Cref{P:main1_bounded} becomes  neater. 

\begin{corollary}\label{C:ineqCompactAcausal}
Let $(X,\sfd, \mm, \ll, \leq, \tau)$ be as in \cref{P:main1_bounded}. Let $V\subset X$ be Borel, achronal and timelike complete subset. 
Let $A\subset I^+(V)$ be any acausal compact set such that 
$A \subset I^-(B)$, for some subset $B\subset X$.
Then 
\begin{equation}\label{E:inequalitygeneralCor}
\mm^{+}(A) \leq \int_{Q} \mm_{\alpha}^{+}( A\cap {X_{\alpha}}) \,\qq(d\alpha), 
\end{equation}
where the disintegration in the integral is given by $\tau_V$ and  Theorem \ref{T:disint}.
\end{corollary}

\begin{proof}
Since $\tau_{V}$ is lower semi-continuous on $I^{+}(V)$ and $A\subset I^+(V)$ is compact, we infer that 
$$
\inf_{x\in A} \tau_V (x) = \min_{x\in A} \tau_V(x)>0.
$$
It remains to prove the existence of an open set $U_0 \supset A$
verifying the hypothesis of \Cref{P:main1_bounded}. 
For each $x \in B$, the set $I^-(x)$ is open and 
by assumption $A\subset \cup_{x\in B} I^-(x)$. 
Therefore by compactness of $A$, 
there exists a finite collection $x_1, \dots, x_n \in B$ 
such that 
$$
A \subset \bigcup_{k =1}^n I^-(x_k).
$$
Since $V \cap I^-(x) \subset V \cap J^-(x)$, the timelike completeness of $V$ implies that the set  $V \cap I^-(x)$
has compact closure in $V$. 
Hence, by global hyperbolicity of $X$, also the set $J^+(V) \cap I^-(x)$ has compact closure. 
In particular, 
$$
\mm(J^+(V) \cap I^-(x_k))<\infty,\quad  \text{for all } k = 1,\dots n.
$$
Moreover, since $\tau$ satisfies the reverse triangle inequality, it follows that 
$$
\sup_{x\in \bigcup_{k =1}^n I^-(x_k)} \tau_V(x)\leq \max_{k=1,\ldots, n} \tau_V(x_k)<\infty.
$$ 
Then, $U_0:=\bigcup_{k =1}^n I^-(x_k)$ fulfills all the assumptions of \Cref{P:main1_bounded}, completing the proof. 
\end{proof}

\begin{remark}[On the assumption $A \subset I^-(B)$]\label{rem:AI-B}
The assumption that there exists a subset   subset $B\subset X$ such that $A \subset I^-(B)$ is clearly satisfied if $X$ is a smooth Lorentzian manifold without boundary. In case $X$ is a smooth Lorentzian manifold with boundary $\partial X$, it suffices that
$$
\{x\in \partial X\colon T_x \partial X \text{ is either null of spacelike}\}\cap A =\emptyset.
$$
In physical term,  this amounts to require that $A$ has empty intersection with the set of spacelike or null singularities (assuming that $X$ is inextendible).
\end{remark}

\section{Lorentzian isoperimetric inequality  and the monotonicity of the area}

Combining the results from the previous sections,  we now derive two geometric inequalities.
We begin discussing the monotonicity of the area.

\subsection{A sharp monotonicity formula for the area of $\tau_V$-level sets}

We prove the following monotonicity formula for the area of the level sets of $\tau_V$ that will be a direct consequence of \Cref{P:identityslice}.

\begin{theorem}[Monotonicity formula for the area]\label{T:monotonicityVolume}
Let  $(X,\sfd, \mm, \ll, \leq, \tau)$ be 
a timelike non-branching,  globally hyperbolic, Lorentzian geodesic space satisfying $\mathsf{TCD}^{e}_{p}(K,N)$ and assume that the causally-reversed structure satisfies the same conditions. 
Let $V\subset X$ be a Borel achronal timelike complete subset.

Denote by $V_{t}$ the achronal slice at $\tau_V$-distance $t$ from $V$, i.e., $V_{t}: = \{\tau_{V} = t\}$. 
Then 
\begin{align*}
&(0,\infty) \ni t \longmapsto \frac{\mm^{+}(V_{t})}{(\fs_{K/(N-1)}(t))^{N-1}} \quad \text{ is monotonically non-increasing, and }\\
&(-\infty,0) \ni t \longmapsto \frac{\mm^{-}(V_{t})}{(\fs_{K/(N-1)}(-t))^{N-1}} \quad \text{ is monotonically non-decreasing}.
\end{align*}
\end{theorem}

In the case $K=0$, i.e., non-negative timelike Ricci (aka Hawking-Penrose strong energy condition), the monotonicity formula takes the following neat expression:
\begin{align*}
&(0,\infty) \ni t \longmapsto \frac{\mm^{+}(V_{t})}{t^{N-1}} \quad \text{ is monotonically non-increasing, and }\\
&(-\infty,0) \ni t \longmapsto \frac{\mm^{-}(V_{t})}{(-t)^{N-1}} \quad \text{ is monotonically non-decreasing.}
\end{align*}

\begin{proof}

From Proposition \ref{P:identityslice} by the continuity of the densities $h(\alpha, \cdot)$ (see (4) in \cref{T:disint}) we deduce that:
$$
\mm^{+}(V_{t}) =\int_{Q_{t}} h(\alpha,t) \,\qq(d\alpha). 
$$
By the second inequality of \eqref{E:MCP0N1d} we have that
$$
h(\alpha,t) \geq \left( \frac{\fs_{K/(N-1)}(t) }{\fs_{K/(N-1)}(T)} \right)^{N-1} h(\alpha,T), \quad \forall \alpha \in Q_{T}, \; \forall T>t>0,
$$
and thus, for all  $T>t>0$:
\begin{align*}
\frac{\mm^{+}(V_{T})}{(\fs_{K/(N-1)}(T))^{N-1}} &=\int_{Q_{T}} \frac{h(\alpha,T)}{(\fs_{K/(N-1)}(T))^{N-1}} \,\qq(d\alpha)\\ 
&\leq  \int_{Q_{T}} \frac{h(\alpha,t)}{(\fs_{K/(N-1)}(t))^{N-1}} \,\qq(d\alpha).
\end{align*}
Since by the very definition, for $T>t>0$, it holds that  $Q_{t} \supset Q_{T}$, we infer that,  for all  $T>t>0$:
$$
\frac{\mm^{+}(V_{T})}{(\fs_{K/(N-1)}(T))^{N-1}} \leq\int_{Q_{t}} \frac{h(\alpha,t)}{(\fs_{K/(N-1)}(t))^{N-1}} \,\qq(d\alpha) =\frac{\mm^{+}(V_{t})}{(\fs_{K/(N-1)}(t))^{N-1}}.
$$
In the symmetric situation of $0> t > T$, we can use the first inequality of \eqref{E:MCP0N1d} to obtain
$$
h(\alpha,-t) \geq \left( \frac{\fs_{K/(N-1)}(-t) }{\fs_{K/(N-1)}(-T)} \right)^{N-1} h(\alpha,-T)
$$
and thus
\begin{align*}
\frac{\mm^{-}(V_{T})}{(\fs_{K/(N-1)}(-T))^{N-1}} &=\int_{Q_{T}} \frac{h(\alpha,T)}{(\fs_{K/(N-1)}(-T))^{N-1}} \,\qq(d\alpha) \nonumber\\
&\leq\int_{Q_{t}} \frac{h(\alpha,t)}{(\fs_{K/(N-1)}(-t))^{N-1}} \,\qq(d\alpha) 
=\frac{\mm^{-}(V_{t})}{(\fs_{K/(N-1)}(-t))^{N-1}}.
\end{align*}
\end{proof}

\begin{remark}[Sharpness of the monotonicity formula in \cref{T:monotonicityVolume}]\label{Rem:SharpMononot}
The area mononoticity in \cref{T:monotonicityVolume} is sharp, as equality is achieved in  conical regions in the model spaces.

More precisely, for $K=0$ and $N\in \mathbb{N}\geq 2$, consider the $N$-dimensional Minkowski space with coordinates $(x_1,\ldots, x_N)$ and Lorentzian metric $dx_1^2+\ldots+dx_{N-1}^2-dx_N^2$. Let $X$ be the conical region 
$$
X:=\{0\}\cup\{(x_1,\ldots, x_N)\colon x_N^2\geq a\,(x_1^2+\ldots+x_{N-1}^2)\}, \quad \text{for some } a>1.
$$
Note that $X$, endowed with the standard metric and Lorentzian structure, is a timelike non-branching Lorentzian geodesic space satisfying $\TCD^e_p(0,N)$. Observe that $X\setminus\{0\}$ is a subset of the open cone of timelike vectors and that $V=\{0\}$ is a Borel achronal timelike complete subset of $X$. A direct computation (see the proof of \cref{prop:SharpnessIsop}) shows that there exists $c=c(a,N)$ such that 
$$
\mm^+(V_t)=c t^{N-1}, \text{ for all } t>0.
$$
In particular, equality is achieved in the monotonocity formula. 
For $K>0$ (resp. $K<0$), one can construct an analogous example replacing the $N$-dimensional Minkowski space by the $N$-dimensional de Sitter space of constant sectional curvature $K/(N-1)$ (resp. the $N$-dimensional anti-de Sitter space of of constant sectional curvature $K/(N-1)$).
\end{remark}

\subsection{A sharp and rigid isoperimetric-type inequality}
We next deduce from \cref{T:local}, \cref{P:main1_bounded} and \cref{C:ineqCompactAcausal} an 
isoperimetric type inequality.

For $V\subset X$, Borel achronal  timelike complete subset, and $S \subset I^{+}(V)$ Borel achronal set we will consider the conically shaped region $C(V,S)$
spanned by the set of $\tau_V$-maximizing geodesics from $V$ to $S$, i.e.: 
\begin{equation}\label{E:Cone}
C(V,S) : = \{ \gamma_{t} \colon \gamma \in \Geo(X), \ t \in [0,1], \ \gamma_{0} \in V, \gamma_{1} \in S, \ {\rm L}_{\tau}(\gamma) = \tau_{V}(\gamma_{1}) \}. 
\end{equation}
Set
\begin{equation}\label{eq:defdist}
\dist(V,S) : = \inf \{ \tau_{V}(x) \colon x \in S \}.
\end{equation}
Notice that, by its very definition, $\dist(V,S)$ is a $\min-\max$ object. 
It shall interpreted as a kind of ``time-distance" between $V$ and $S$.
Define also 
\begin{equation}
{\mathfrak D}_{K,N}(t):= \frac{1}{\fs_{K/(N-1)} (t)^{N-1}} \int_{0}^{t}  \fs_{K/(N-1)} (s)^{N-1} \, ds, \quad t\in (0, T_{K,N})
\end{equation}
where $\fs_{K/(N-1)} (t)$ was defined in \eqref{eq:deffsfc} and 
$$T_{K,N}:=\sup\{t>0: \, \fs_{K/(N-1)} (t) >0\}\in (0,\infty].$$ Note that, for $K=0$, one obtains simply $${\mathfrak D}_{0,N}(t)= \frac{t}{N}.
$$
The function ${\mathfrak D}_{K,N}$ admits an equivalent expression: since 
$$
\sigma_{K/(N-1)}^{(s/t)}(t)= \frac{\fs_{K/(N-1)} (s)}{\fs_{K/(N-1)} (t)},
$$
it follows by a change of variable that
$$
{\mathfrak D}_{K,N}(t)
= \int_0^t
\sigma_{K/(N-1)}^{(s/t)}(t)^{N-1} \,ds
= t \int_0^1 
\sigma_{K/(N-1)}^{(r)}(t)^{N-1} \,dr.
$$
\begin{theorem}[Isoperimetric-type inequality]\label{T:isop1} 
Let  $(X,\sfd, \mm, \ll, \leq, \tau)$ be 
a timelike non-branching,  globally hyperbolic, Lorentzian geodesic space satisfying $\mathsf{TCD}^{e}_{p}(K,N)$, and assume that the causally-reversed structure satisfies the same conditions.

Let $V\subset X$ be a Borel achronal  timelike complete 
subset and $S \subset I^{+}(V)$ be a compact and acausal  set such that $S \subset I^{-}(B)$ for some $B\subset X$. 
Then
\begin{equation}\label{E:isoperKN}
\mm^{+}(S)\, {\mathfrak D}_{K,N}(\dist(V,S)) \leq  \mm(C(V,S)),
\end{equation}
If $K=0$,  the bound \eqref{E:isoperKN} reads as
\begin{equation}\label{E:isoper}
\mm^{+}(S)\, \dist(V,S) \leq N \mm(C(V,S)).
\end{equation}
\end{theorem}

\begin{proof}
Consider the disintegration formula associated to $\tau_{V}$. Since $S \subset I^{+}(V)$, then $C(V,S) \subset I^{+}(V)$  and therefore
$$
\mm(C(V,S)) = \int_{Q} \mm_{\alpha} (X_{\alpha} \cap C(V,S)) \,\qq(d\alpha). 
$$
Since both $V$ and $S$ are achronal,   $X_{\alpha} \cap C(V,S)$ can be identified via $\tau_{V}$  to a real interval $[0,b_{\alpha}]$ for some $b_{\alpha} > 0$. 
Then 
\begin{equation}\label{eq:VolCVS}
\mm(C(V,S)) = \int_{Q} \int_{[0,b_{\alpha}]} h(\alpha,s)\,ds \,\qq(d\alpha).
\end{equation}
By \eqref{E:MCP0N1d}, we have that  for $\qq$-a.e.\;$\alpha \in Q$ it holds that $b_\alpha\in (0,T_{K,N})$; moreover, if $0< s < b_{\alpha}$, then
\begin{equation}\label{eq:LBhalpha}
h(\alpha,s) \geq h(\alpha,b_{\alpha}) \, \frac{\fs_{K/(N-1)} (s)^{N-1}}{ \fs_{K/(N-1)} (b_\alpha)^{N-1}}.
\end{equation}
Hence,
$$
\int_{0}^{b_{\alpha}} h(\alpha,s)\,ds \geq  \frac{h(\alpha,b_{\alpha})}{\fs_{K/(N-1)} (b_\alpha)^{N-1}} \int_{0}^{b_{\alpha}} \fs_{K/(N-1)} (s)^{N-1} ds = h(\alpha,b_{\alpha})\, {\mathfrak D}_{K,N}(b_\alpha).
$$
Notice that the function $(0, T_{K,N})\ni t\mapsto {\mathfrak D}_{K,N}(t)$ is increasing. Moreover, if we denote by $\{ z_{\alpha}\} = X_{\alpha}\cap S$, then $b_{\alpha} = \tau_{V}(z_{\alpha})$, yielding $b_{\alpha} \geq \dist(V,S)$.
We infer that 
\begin{equation}\label{eq:LBhalphaDist}
\int_{0}^{b_{\alpha}} h(\alpha,s)\,ds  \geq h(\alpha,b_{\alpha}) \, {\mathfrak D}_{K,N}(\dist(V,S)), \quad \qq\text{-a.e. }\alpha \in Q.
\end{equation}
The combination of \eqref{eq:VolCVS} and \eqref{eq:LBhalphaDist} gives
\begin{align*}
\mm(C(V,S)) &\geq {\mathfrak D}_{K,N}(\dist(V,S))\, \int_{Q} h(\alpha,b_{\alpha}) \, \qq(d\alpha) \\
&=  {\mathfrak D}_{K,N}(\dist(V,S))\,  \int_{Q} \mm_{\alpha}^{+}(S\cap X_{\alpha}) \, \qq(d\alpha),
\end{align*}
which, together with \cref{C:ineqCompactAcausal}, concludes the proof of the inequality.
\end{proof}

\begin{remark}\label{R:otherinequalities}
\cref{T:isop1} is valid also for other classes of sets $S$, in particular those for which the inequality of \cref{C:ineqCompactAcausal} holds true. 
In particular, if (see \cref{P:main1_bounded})
\begin{itemize}
\item $S \subset I^{+}(V)$ is a closed, acuasal set such that $\inf_{x \in S} \tau_{V}(x) > 0$  and there exists an open set $U_{0}$ such that 
$\sup_{v\in U_0} \tau_V(x)<\infty$ 
(in case $K\geq 0$, the finiteness of the $\sup$ is not needed),
$S \subset U_0$ and $\mm(J^-(U_0)\cap J^+(V))<\infty$;
\end{itemize}
or  (see \cref{P:main1})
\begin{itemize}
\item $S$ is Borel, achronal  with $\partial_V^+S = \emptyset$  and $\inf_{x\in S} \tau_V(x)>0$; 
\end{itemize} 
Then 
$$\mm^{+}(S)\, {\mathfrak D}_{K,N}(\dist(V,S)) \leq  \mm(C(V,S)).$$
\end{remark}

In the next corollary, we specialize  \cref{T:isop1} to the case of a smooth Lorentzian manifold (recall also \cref{rem:AI-B}).

\begin{corollary}\label{cor:IsopSmooth}
Let $(M^{n+1},g)$ be a smooth globally hyperbolic Lorentzian manifold. Assume there exists $K\in \R$ such that  $\Ric_g(v,v)\geq -K g(v,v)$ for all timelike tangent vectors.

Let $V\subset X$ be a Borel achronal  timelike complete 
subset and $S \subset I^{+}(V)$ be a compact and acausal  smooth hypersurface.

In case $M$ has non-empty boundary $\partial M$, assume that
$$
\{x\in \partial M\colon T_x \partial M \text{ is either null of spacelike}\}\cap S =\emptyset.
$$

 Then
\begin{equation}\label{E:isoperKNCor}
\vol_g^n(S)\, {\mathfrak D}_{K,n+1}(\dist(V,S)) \leq  \vol_g^{n+1}(C(V,S)),
\end{equation}
where $\vol_g^{n+1}$ (resp.\;$\vol_g^n$) denotes the $(n+1)$-dimensional Lebesgue measure of $g$ (resp.\;the  $n$-dimensional Lebesgue measure of the restriction of $g$ to $S$).  

If $K=0$ (i.e., if the strong energy condition holds),  the bound \eqref{E:isoperKNCor} reads as
\begin{equation}\label{E:isoperCor}
\vol_g^n(S)\, \dist(V,S) \leq (n+1)\; \vol_g^{n+1}(C(V,S)).
\end{equation}
\end{corollary}

\begin{remark}[Related literature in Riemannian signature]
At a formal level, the proof of Theorem \ref{T:isop1} is performed following the integral lines of the gradient flow  of $\tau_V$, the Lorentzian distance from $V$. This should be compared with the celebrated Heintze-Karcher inequality in the Riemannian setting \cite{HeintzeKarcher}, where one obtains a volume bound of a smooth Riemannian manifold $M$ in terms of the co-dimensional one volume of a smooth hypersurface $V$, the maximal value of the mean curvature of $V$ and the maximal distance from $V$ in $M$. The proof of the Heintze-Karcher inequality is also performed following the integral lines of the gradient flow of the distance from $V$, however the volume bound on each integral line depends on the mean curvature of $V$. The main advantage of the proof of  Theorem \ref{T:isop1} is that, in addition to considerably relaxing the regularity assumed on the space, it does not assume any bound on the mean curvature of $V$.

Instead, the statement of Theorem \ref{T:isop1} is more in the spirit of the isoperimetric-isodiametric inequalities studied by the second named author and Spadaro \cite{MondinoSpadaro} in Riemannian signature, with different techniques.
\end{remark}

We now show that \cref{T:isop1} is sharp.

\begin{proposition}[Sharpness of \cref{T:isop1}]\label{prop:SharpnessIsop}
The inequality \eqref{E:isoper} is sharp, in the following sense. For $N\in \mathbb{N}$, $N\geq 2$:
\begin{itemize}
\item for $K=0$, the equality in \eqref{E:isoper} is achieved for a conical region in $N$-dimensional Minkowski spacetime.
\item for $K>0$, the equality  in \eqref{E:isoperKN} is achieved for a conical region in $N$-dimensional de Sitter spacetime with constant sectional curvature $K/(N-1)$;
\item for $K<0$, the equality  in \eqref{E:isoperKN} is achieved for a conical region in $N$-dimensional anti-de Sitter spacetime with constant sectional curvature $K/(N-1)$;
\end{itemize} 
\end{proposition}

\begin{proof}
We will first consider the two dimensional Minkowski spacetime $\mathbb{M}^{2}$ with metric $-dy^{2} + dx^{2}$ and reference measure the volume measure, i.e., the two dimensional Lebesgue measure. 
Consider the set $$S = \{ (x,y) \in \R^{2} \colon  -y^{2} + x^{2} = -1 \}$$ and  restrict the space to 
$$
X : = \left\{ (x,y) \colon -a \leq x \leq a, \, y \geq |x| \frac{\sqrt{1+a^{2}}}{a} \right\},
$$
where $a>0$ is any fixed positive constant.
Taking $V = \{ (0,0) \}$, then 
$$
C(V,S) = \{  (x,y) \in X \colon -y^{2} + x^{2} \leq 1 \}, \qquad S = \{ (x, \sqrt{1+x^{2}}) \colon x \in (-a,a) \}.
$$
It is straightforward to compute
\begin{align*}
\mathcal{L}^{2}(C(V,S)) &= 2 \left( \int_{(0,a)} \sqrt{1+x^{2}} \,dx  - \frac{a\sqrt{1+a^{2}}}{2}\right) \\&= \left.  (x\sqrt{1+x^{2}} +\sinh^{-1}(x)) \right|^{a}_{0} - a\sqrt{1+a^{2}}.
\end{align*}
The length $\ell$ of $S$ is  given by 
$$
\ell = 2 \int_{0}^{a} \sqrt{1 -y'(x)^{2}}\,dx = 2 \int_{0}^{a} \frac{1}{\sqrt{1+x^{2}}}\,dx = 2 \left. \sinh^{-1}(x) \right|^{a}_{0}.
$$
Since by construction $\dist(V,S) = 1$, for this example, the isoperimetric type inequality \eqref{E:isoper} becomes an identity for $N = 2$.
\smallskip

In higher dimension $N=n+1 \geq 3$, we consider the $n+1$-dimensional Minkowski space $\mathbb{M}^{n+1}$ with metric $g = -dt^{2} + dx_{1}^{2}+\ldots+ dx_{n}^{2}$. 
Consider the cone 
$$
X = \{ (x,t) \colon \|x \| \leq a, t  \geq \| x\| \sqrt{1 +a^{2}}/a \}, \quad \text{for any fixed } a > 0,
$$
and the surface 
\begin{equation}\label{eq:defSSharp}
S = \{ (x,t) \in X  \colon  t^{2}  -
\|x\|^{2} = 1\}.
\end{equation}
The achronal set $V$ will be the origin $O$.

The volume of $C(V,S)$ will be the difference between the volume of the cone 
$$W: = X \cap \{ 0 \leq t \leq  \sqrt{1+a^{2}}\}$$
and the volume of $E$, the epigraph in $W$ of the function $t = \sqrt{1+ \| x\|^{2}}$. 
Then 
\begin{align*}
\mathcal{L}^{n+1} (W) &~ = \int_{0}^{\sqrt{1+a^{2}}} \omega_{n} \left(\frac{a}{\sqrt{1+a^{2}}}\right)^{n} r^{n} \,dr =  \frac{\omega_{n}}{n+1}\left(\frac{a}{\sqrt{1+a^{2}}}\right)^{n} \sqrt{1+a^{2}}^{n+1} \\
&~ = \sqrt{1+a^{2}} a^{n} \frac{\omega_{n}}{n+1},\\
\mathcal{L}^{n+1} (E) &~ = \omega_{n} \int_{1}^{\sqrt{1+a^{2}}} (\sqrt{r^{2}-1})^{n} dr.
\end{align*}
Thus
$$
\mathcal{L}^{n+1} (C(V,S))=\omega_n \left( \frac{a^n}{n+1} \sqrt{1+a^{2}}- \int_{1}^{\sqrt{1+a^{2}}} (\sqrt{r^{2}-1})^{n} dr \right). 
$$
For computing the area of $S$, we parametrize $S$ via the graph of the function $\sqrt{1 +r^{2}}$
over the polar coordinates $(r,\Theta) \in [0,\infty)\times \mathbb{S}^{n-1}$ in $\R^{n}$.
The tangent space of $S$ is spanned by $\partial_{t}$ and $\partial_{\Theta_{i}}, i=1,\dots, n-1$. Restricting the Minkowski metric to $S$, the area form is given  by $r^{n-1}/\sqrt{1+r^{2}}\, dr d\Theta$.
Then 
\begin{align*}
{\rm Area}(S) & = \int_{\mathbb{S}^{n-1}} \int_{0}^{a} r^{n-1}/\sqrt{1+r^{2}} \,dr \,d\Theta =  n \omega_{n} \int_{0}^{a} r^{n-1}/\sqrt{1+r^{2}} \,dr\\
&= n\omega_n \int_1^{\sqrt{1+a^2}} (\sqrt{x^2-1})^{n-2} dx,
\end{align*}
where, in the last identity, we performed the change of variables $x=\sqrt{1+r^2}$.
\\It is possible to check that,  for all $a>0$, $n\geq 2$:
\begin{equation}\label{eq:claimSharp}
  \int_1^{\sqrt{1+a^2}} n (\sqrt{x^2-1})^{n-2} + (n+1) (\sqrt{x^2-1})^{n} dx= a^n \sqrt{1+a^2},
\end{equation}
yielding
$${\rm Area}(S)=(n+1) \mathcal{L}^{n+1} (C(V,S)),\quad \text{for all } a>0, \,n\geq 2.$$
Since, by construction, all the points in $S$ are at distance 1 from the origin $O$, we just showed that $S$ defined in \eqref{eq:defSSharp} achieves the equality in \eqref{E:isoper} for $V=\{O\}$.
\\This shows sharpness for $K=0$, $N\in \mathbb{N}, N\geq 2$. 

For $K\neq 0$, $N\in \mathbb{N}, N\geq 2$, up to scaling we can assume that $K=N-1$ (if $K>0$) or  $K=-(N-1)$  (if $K<0$). One can check that  equality in \eqref{E:isoperKN} is achieved by the following choices.  In the arguments above, replace the Minkowski space by the de Sitter space (in case $K=N-1$) or by the anti-de Sitter space (in case $K=-(N-1)$), the cone $X$ by the exponential of $\exp_p(X)$, the surface $S$ by  $\exp_p(S)$, the domain $W$ by $\exp_p(W)$ and set $V=\{p\}$.
\end{proof}

We next show that \cref{T:isop1} is also rigid. 
\begin{proposition}[Rigidity of \cref{T:isop1}]\label{prop:RigidityIsop}
The inequality \eqref{E:isoperKN} is rigid, in the following sense. In addition to the assumptions of  \cref{T:isop1}, assume that
\begin{enumerate}
    \item[(i)] $S$ is a smooth spacelike hypersurface;
    \item[(ii)] $C(V,S)\setminus V$ is isometric to a smooth Lorentzian manifold $(M^{n+1},g)$, incomplete along $V$ and with boundary $S$;
    \item[(iii)] Equality is achieved in \eqref{E:isoperKN}, namely
\begin{equation}\label{E:isoperKN=}
{\rm Vol}_g^{n}(S)\, {\mathfrak D}_{K,N}(\dist(V,S)) =  {\rm Vol}_g^{n+1}(C(V,S)),
\end{equation}
where ${\rm Vol}_g^{n+1}$ (resp. ${\rm Vol}_g^{n}$) denotes the $(n+1)$-dimensional volume measure associated to $g$ (resp.\;the $n$-dimensional volume measure associated to the restriction of $g$).
\end{enumerate}
Then
\begin{enumerate}
\item[(a)] $V=\{\bar x\}$ is a singleton; denote by $g_S:= \lambda^{-2} g\llcorner TS$ the normalised restriction of $g$ to $TS$, where $\lambda:=\dist(V,S)^{-2}$ is a normalization constant;
\item[(b)] Let 
\begin{equation}\label{eq:defC(S)}
{\rm C}(S):=[0,\dist(V,S)]\times S/ \sim , \text{ where } (0,x)\sim (0,y) \text{ for all } x,y\in S,
\end{equation}
be a (truncated) cone over $S$ and endow it with the Lorentzian metric (defined outside the tip $\{r=0\}$)
\begin{equation}\label{eq:defgC(S)}
    g_{{\rm C}(S)}:=-dr^2+ \fs_{K/n}(r)^2  g_S,
\end{equation}
where  we use the notation  $(r,x)\in [0,\dist(V,S)]\times S/\sim$, and  $\fs_{(\cdot)}(\cdot)$ is defined in \eqref{eq:deffsfc}.

 Then there exists an isometry $\Psi:  {\rm C}(S) \to  C(V,S) $ such that $\Psi(\{r=0\})=V$ is the tip of the cone and $\Psi(\{r= \dist(V,S)\})=S$. 
In particular, $S$ has constant mean curvature.
\end{enumerate} 

If $(ii)$ is replaced by the stronger
\begin{enumerate}
\item[(ii')] $C(V,S)$ is contained in a smooth Lorentzian manifold $(M^{n+1},g)$ complete and without boundary,
\end{enumerate}
then $(b)$ can be improved into
\begin{enumerate}
\item[(b')] $(S,g_S)$ is isometric to a subset of the $n$-dimensional hyperbolic space of constant curvature $-1$, $(\mathbb{H}^n, g_{\mathbb{H}^n})$, and $C(V,S)$ is isometric to a cone in the model space with metric $-dr^2+\fs_{K/n}(r)^2  g_{\mathbb{H}^n}$ (note that $K=0$ gives Minkowski, $K=n$  de Sitter, and $K=-n$ anti-de Sitter)   and with tip at $V=\{\bar{x}\}$. 
\end{enumerate}
\end{proposition}

\begin{proof}
For the sake of brevity we only sketch the proof.

Following the proof of \cref{T:isop1}, it is clear that  equality in \eqref{E:isoperKN} forces equality in \eqref{eq:LBhalpha} for $\qq$-a.e. $\alpha$, namely:
\begin{equation}\label{eq:LBhalpha=}
\frac{h(\alpha,s)} {h(\alpha,b_{\alpha})} =   \frac{\fs_{K/(N-1)} (s)^{N-1}}{ \fs_{K/(N-1)} (b_\alpha)^{N-1}}, \quad \text{for $\qq$-a.e. $\alpha$, for all $s\in [0,b_\alpha]$}.
\end{equation}
Moreover, the fact that $(0, T_{K,N})\ni t\mapsto {\mathfrak D}_{K,N}(t)$ is increasing forces 
\begin{equation}\label{eq:balpha=dist}
b_\alpha=\dist(V,S), \quad \text{for $\qq$-a.e. $\alpha$}.
\end{equation}
This means that (up to a set of $\qq$-measure zero) one can identify $S$ with the quotient set $Q$ and parametrize $C(V,S)$, up to a set of $\mm$-measure zero, by the ray map 
\begin{equation}\label{eq:defPsi}
\Psi:[0,\dist(V,S)]\times S\to C(V,S), \quad \Psi(s,\alpha):=X_\alpha(s),
\end{equation}
so that $\Psi$ is a Borel bijection (up to a set of $\mm$-measure zero). Notice that, by construction, $\Psi(\{0\}\times S)=V$ and $\Psi(\{\dist(V,S)\}\times S)=S$. Moreover,  \eqref{eq:LBhalpha=} yields that the co-dimension one volume of the $s$-section $\Psi(\{s\}\times S)$ tends to $0$ as $s\to 0$.
\\Now, using the smoothness assumption on $C(V,S)$ and standard Jacobi fields computations, one can mimic the proof of the rigidity in Bishop-Gromov inequality  (see for instance \cite{EhrlichSanchez, Chav06}) in order to infer that $\Psi$ defined in \eqref{eq:defPsi} passes to the quotient \eqref{eq:defC(S)} and defines an isometry between $(C(V,S), g)$ and $({\rm C}(S), g_{{\rm C}(S)})$. This shows $(a)$ and $(b)$. In order to prove $(b')$, it suffices to apply \cite[Thm.\;2.2]{EhrlichSanchez}.
\end{proof}

In the next remark, we discuss a direct application of the isoperimetric inequality \eqref{E:isoper}. 

\begin{remark}[An upper bound on the area of acausal hypersurfaces in a black hole interior]\label{rem:AreaBoundCauchyBlackHole}
Let $(M^{4},g)$ be a globally hyperbolic spacetime of finite volume  and satisfying the Strong Energy Condition (SEC for short, i.e., $\Ric(v,v)\geq 0$ for all $v\in TM$ timelike). $M$ shall be thought as a finite slab in the interior of a black hole.  Of course any black hole metric satisfying the vacuum Einstein equations $ \Ric \equiv 0$ (such as Schwarzschild or Kerr) also satisfies the SEC.

Assume there exists a subset  $\Sigma\subset M$ achronal and past complete. It is natural to expect that the ``singular set at the center of the black hole" satisfies such properties (when $M$ is the black hole interior), at least for a generic black hole.

Let $S\subset I^-(\Sigma)$ be
\begin{itemize}
    \item Either a compact and acausal smooth hypersurface, with $\partial M\cap S=\emptyset$;
    \item or a (possibly unbounded) smooth achronal hypersurface with $\partial^-_\Sigma S=\emptyset$; in particular this holds if $S$ is a Cauchy hypersurface.
\end{itemize}
The quantity $\dist(S, \Sigma)$ shall be thought as the Lorentzian distance from the hypersurface $S$ to the singular set $\Sigma$.

Applying Theorem \ref{T:isop1} (see also \cref{R:otherinequalities} and \cref{cor:IsopSmooth}) to the causally reversed structure (i.e., backward in time), we obtain 
\begin{equation}\label{E:isoperBH}
{\rm Vol}^3_{g} (S)\, \dist(S, \Sigma) \leq 4 \, {\rm Vol}^4_g(M),
\end{equation}
giving an upper bound on the area of the  hypersurface $S$ (with respect to the $3$-dimensional volume measure ${\rm Vol}^3_{g}$ associated to the restriction of $g$ to $S$) in terms of its time-distance from the singular set $\Sigma$ and the volume ${\rm Vol}^4_g(M)$ (of the slab) of the black hole interior $M$  (with respect to the $4$-dimensional volume measure ${\rm Vol}^4_{g}$ associated to $g$).
\end{remark}

\begin{example}[An upper bound on the area of acausal hypersurfaces in the Schwarzschild black hole interior]\label{Example:SchwInterior}
To fix the ideas by an explicit example, consider a finite slab $\{t\in [a,b]\}$  in the interior  of the Schwarzschild black hole
\begin{equation}\label{eq:defScwSlab}
M:=\{t\in [a,b]\} \cap \{r\leq 2m\} 
\end{equation}
endowed, in region  $ \{0<r< 2m\}$,  with the  metric
\begin{equation}\label{eq:gschw}
g:=-\left(1-\frac{2m}{r}\right) dt^2+ \left(1-\frac{2m}{r}\right)^{-1} dr^2 + r^2(d\theta^2 +\sin^2\theta \, d\varphi^2)
\end{equation}
and with time orientation so that $-\frac{\partial}{\partial r}$ is future oriented. It is well-known that the singularity at $\{r=2m\}$ is just apparent, in the sense that the metric $g$ is perfectly smooth after a suitable coordinate change; on the contrary, at $\{r=0\}$ the singularity is at the $C^0$-level (see \cite{Sbierski}). 
Note that in the black hole interior, $r$ is a timelike variable while $t$ is a spacelike variable.
It is clear that the singular set $\Sigma:=\{r=0, t\in [a,b]\}$ is achronal. 
\\ Let  $S \subset \{0<r< 2m\}$ be a smooth compact acausal  hypersurface; a (trivial) example of such a hypersurface is $\{r=r_0\}$, $r_0\in (0, 2m)$.

Repeating the proof of Theorem \ref{T:isop1} to the causally reversed structure (i.e., backward in time) as in  \eqref{E:isoperBH}, we obtain
\begin{equation}\label{E:isoperSch}
{\rm Vol}^3_{g} (S)\,   \inf_S \tau_{\Sigma}(r)  \leq 4 {\rm Vol}^4_{g}(\{r\leq 2m, t\in [a,b]\})= \frac{128}{3} \pi m^3 (b-a),
\end{equation}
where $\tau_{\Sigma}$ depends only on the $r$-coordinate and is given by the expression 
\begin{equation}\label{eq:tau(r)Schw}
\tau_{\Sigma}(r)=\pi m - \sqrt{2mr-r^2}- 2m \arctan \left( \sqrt \frac{2m-r}{r} \right).
\end{equation}
Notice that $\tau_{\Sigma}(r_0)$ is the maximal proper time that may lapse for a massive observer initially at $r=r_0\in (0,2m]$ before hitting the singularity $\Sigma$. 

While well-known to experts, let us briefly sketch the proof of the expression \eqref{eq:tau(r)Schw} for completeness of presentation.
Let $\gamma_\tau=(t(\tau), r(\tau), \theta(\tau), \varphi(\tau))$, $\tau\in [0,T]$, 
be a future directed timelike curve parametrized by proper time $\tau$:
\begin{equation*}
1= \left(\frac{2m}{r}-1\right)^{-1} \left(\frac{d r} {d \tau}\right)^2- \left(\frac{2m}{r}-1\right) \left( \frac{d t} {d \tau} \right)^2 -r^2 \left(\frac{d \theta} {d \tau} \right)^2 - r^2 \sin^2 \theta   \left(\frac{d \varphi} {d \tau} \right)^2.
\end{equation*}
 From the fact that $\gamma$ is future directed, we infer that $\tau\mapsto r(\tau)$ is strictly decreasing and 
\begin{equation}\label{eq:drdtaugeq}
\frac{d r} {d \tau} \leq - \sqrt{\frac{2m}{r}-1},
\end{equation}
with equality if and only if 
\begin{equation}\label{spacial=0}
\left(\frac{2m}{r}-1\right) \left( \frac{d t} {d \tau} \right)^2 +r^2 \left(\frac{d \theta} {d \tau} \right)^2 + r^2 \sin^2 \theta   \left(\frac{d \varphi} {d \tau} \right)^2=0.
\end{equation}
Recall that the aim here is, given $\gamma_0=(t_0, r_0, \theta_0, \varphi_0)\in M$, find (if it exists) the future timelike curve $(\gamma_\tau)_{\tau\in [0,T]}$ with $\gamma_{T}\in \Sigma$ (i.e., $r(\gamma_T)=0$) parametrized by proper time, and having the maximal $T$.  This amounts to find the future directed timelike curve having maximal $\frac{dr}{d \tau}$. From \eqref{eq:drdtaugeq} and \eqref{spacial=0}, it is clear that such a curve has to be radial, i.e., $\gamma_\tau=(t_0, r(\tau), \theta_0, \varphi_0)$, and such that
$$
T=\int_0^{r_0} \left(\frac{2m}{r}-1\right)^{-1/2} \,dr
= \pi m - \sqrt{2mr_0-r_0^2}- 2m \arctan \left( \sqrt{\frac{2m-r_0}{r_0}}  \right).
$$
This completes the proof of \eqref{eq:tau(r)Schw}.

One can deduce analogous bounds for de-Sitter Schwarzschild and anti de Sitter-Schwarz\-schild black holes, by using the more general \eqref{E:isoperKN}.
\end{example}

\begin{remark} [An upper bound on the area of acausal hypersurfaces in cosmological spacetimes]\label{rem:CauchyCosmological}
Another situation where Theorem \ref{T:isop1} seems to give some new geometric information,  is for cosmological spacetimes. In this case, the manifold is homeomorphic to a Lorentzian cone $C(\Sigma)$ over a manifold $\Sigma$, with coordinates $(x,t), t\geq 0, x\in \Sigma$ (and diffeomorphic on the open subset $\{t>0\}$), with $\Sigma\times\{t\}$ spacelike slices for any $t>0$ and with $\frac{\partial}{\partial t}$ timelike. In such a model, the point $\{t=0\}$ corresponds to the origin of the universe,  i.e., the  ``big-bang".

In Theorem  \ref{T:isop1} (see also \cref{R:otherinequalities}), we can choose $V=\{t=0\}$  and  $S\subset\{t>0\}$ such that
\begin{itemize}
    \item either $S$ is a compact and acausal smooth hypersurface;
    \item or $S$ is a (possibly unbounded) smooth achronal hypersurface with $\partial^+_V S=\emptyset$; in particular this holds if $S$ is a Cauchy hypersurface.
\end{itemize}

The quantity $\dist(V,S)$ could be loosely interpreted as a kind of  ``age" of $S$, while the lower bound $K$ on the timelike Ricci curvature is related to the cosmologiocal constant and the energy momentum tensor via the Einstein equations.
\\In previous literature  \cite[Thm.\;2 and Prop.\;3]{ACKW09}, area bounds on  achronal spacelike hypersurfaces were proved in Friedman-LeMa\^itre-Robertson-Walker spacetimes with non-negative  timelike Ricci curvature (see also \cite{Flaim} for related results). These are  cosmological spacetimes where the time-slices  $\Sigma\times\{t\}$ have constant sectional curvature for all $t>0$, i.e., are homogenous and isotropic. 
Although such symmetries are satisfied at a very good level of approximation at the scale of the universe,  recent observations detected some anomalies in the cosmic microwave background that are challenging such a model (see for instance \cite{ESA}). Since Theorem \ref{T:isop1} does not assume any symmetry and allows any $K\in \R$, it  gives an area bound on $S$ also in the case when the timelike Ricci curvature is bounded below by a negative constant $K$ (thus allowing more freedom to the cosmological constant and to the energy-momentum tensor), and the  time slices $\Sigma\times\{t\},\,  t>0,$ are not necessarily homogenous and isotropic. 
\end{remark}

\section{Further Localization results}
\label{S:Further}
For completeness, we include a brief discussion on a generalisation of Theorem \ref{T:local}.
The results of  \cref{Ss:transportrelation} and  \cref{Ss:disintegrationregularity}
are indeed valid for a wider class of functions then time separation functions from 
achronal timelike complete sets, namely for solutions of the dual Kantorovich problem for $p = 1$. 
In analogy with the metric theory, this class coincides with the class 
of timelike reverse 1-Lipschitz functions defined as follows: 
$$
u : X \to \R, \qquad u(y) - u(x) \geq \ell(x,y), \quad \forall \ x,y \in X. 
$$
For a  fixed timelike reverse 1-Lipschitz function $u$,
define the transport relation as 
$$
\Gamma_{u} : = \{(x,y) \in X_{\leq} \colon u(y) - u(x) = \tau(x,y) \}.
$$
One can check that $\Gamma_{u}$ is $\ell$-cyclically monotone:
for any $n\in \N$ and any family $(x_{1}, y_{1}), \ldots, (x_{n}, y_{n})$ of points in $\Gamma_{u}$: 
\begin{align*}
\sum_{i=1}^{n}\tau(x_{i}, y_{i})
=&~ \sum_{i=1}^{n} u(y_{i}) -u(x_{i})  \\
=&~ \sum_{i=1}^{n} u(y_{i +1 }) -u(x_{i})  \geq \sum_{i=1}^{N}\ell(x_{i+1}, y_{i}). 
\end{align*}

It is then natural to define $R_{u}, \T_{u}^{end}, \fa(\T_{u}^{end}), \fb(\T_{u}^{end})$ and $\T_{u}$ as in \eqref{E:transport}, \eqref{eq:defendpoints}, \eqref{E:nbtransport} 
with the replacement of $\Gamma_{V}$ by $\Gamma_{u}$.  
By the timelike non-branching property, $\mathcal{T}_{u}$ is partitioned by transport rays induced by $u$,
precisely as for $\tau_{V}$.
Then \cref{P:cpgeod}
can be applied to $\Gamma_{u}$  
to obtain, repeating verbatim the calculations done for $\tau_{V}$,  
the following disintegration result for $\mm$.

Notice that however we do not make any claim on the size of $\mathcal{T}_u$ which, possibly depending on the function $u$,  might be even empty.

\begin{theorem}\label{T:generalDisintegration}
Let  $(X,\sfd, \mm, \ll, \leq, \tau)$ be a timelike non-branching,  globally hyperbolic, Lorentzian geodesic space satisfying $\mathsf{TCD}^{e}_{p}(K,N)$ and assume that the causally-reversed structure satisfies the same conditions.\\ Let $u : X \to \R$ be a timelike reverse $1$-Lipschitz function. 
\\Then $\mm(\fa(\T_{u}^{end}))=\mm(\fb(\T_{u}^{end})=0$ and the following disintegration formula holds true: 
\begin{equation}\label{E:disintegrationLip}
\mm\llcorner_{\T^{end}_{u}} = 
\mm\llcorner_{\T_{u}} 
= \int_{Q} \mm_{\alpha}\, \qq(d\alpha)= \int_{Q} h(\alpha,\cdot) \, \L^{1}\llcorner_{X_{\alpha}}\, \qq(d\alpha),
\end{equation}
where
\begin{itemize}
\item $\qq$ is a probability measure over the Borel quotient set $Q \subset \T_{u}$;
\item  $h(\alpha,\cdot)\in L^{1}_{loc}(X_{\alpha}, \L^{1}\llcorner_{X_{\alpha}})$ for $\qq$-a.e. $\alpha\in Q$;
\item  the map 
$\alpha \mapsto \mm_{\alpha}(A)= h(\alpha,\cdot)\L^{1}\llcorner_{X_{\alpha}}(A)$ is 
$\qq$-measurable for every Borel set $A \subset \T_{V}$.
\item For $\qq$-a.e. $\alpha$ the one-dimensional metric measure space 
$(X_{\alpha},|\cdot|, \mm_{\alpha})$ satisfies the classical $\CD(K,N)$ condition, i.e.\;\eqref{eq:DiffIneqCDKN} holds.
\end{itemize}
\end{theorem}

Finally let us note notice that 
if one replaces the $\TCD^{e}_{p}(K,N)$ assumption by the weaker 
$\TMCP^{e}(K,N)$, then 
all the claims of Theorem \ref{T:generalDisintegration} remain valid except from the last point 
that has to be replaced by ``$(X_{\alpha},|\cdot|, \mm_{\alpha})$ satisfies the classical $\MCP(K,N)$, i.e.\;\eqref{E:MCP0N1d} holds.''

\subsection{Sharp Brunn-Minkowski inequality via localization}

The metric version of Theorem \ref{T:generalDisintegration} 
goes back to \cite{CM1} and, previously, the Riemannian version to \cite{klartag}. 
There, many geometric inequalities (in their sharp form) where obtained by applying  
the theorem to a special function $u$ associated to an optimal transport problem 
induced by a Borel function $f : X \to \R$ having zero mean $\int_{X}f \, \mm = 0$: 
the function $u$ was a Kantorovich potential for the $W_{1}$ optimal transport problem between the 
measures $\mu_{0} : = f^{+}\, \mm$ and $\mu_{1} : = f^{-}\,\mm$, where $f^{\pm}$ are the positive and the 
negative part of $f$.
In this case, the disintegration induced by $u$ localises also the zero mean property of $f$: 
for $\qq$-a.e. $\alpha \in Q$
$$
\int_{X_{\alpha}} f \, \mm_{\alpha} = 0.
$$

It would be therefore desirable to include in the list of properties of $\mm_{\alpha}$ of  
\cref{T:generalDisintegration}, in the case $u$ is the Kantorovich potential associated to $f$, 
also the localization of the zero mean condition. 

It will be obtained by directly adapting the argument of \cite{CM1} 
with the additional difficulty that we cannot rely on the existence of a solution 
to the dual Kantorovich potential.

For ease of notation, we will drop the normalisation factor by directly assuming that $\mu_{0}$ and $\mu_{1}$ are probability measures.

\begin{theorem}\label{T:localization}
Let  $(X,\sfd, \mm, \ll, \leq, \tau)$ be 
a timelike non-branching,  globally hyperbolic, Lorentzian geodesic space satisfying $\mathsf{TCD}^{e}_{p}(K,N)$ and assume that the causally-reversed structure satisfies the same conditions. 

Moreover assume that $\int f \, \mm = 0$ for some real valued Borel function $f$ and 
that the pair of probability measure $(\mu_{0} : = f^{+}\, \mm, \mu_{1} : = f^{-}\, \mm)$ is 
timelike $1$-dualisable. 
Then there exists an $\ell$-cyclically monotone set $\Gamma_{f}$  inducing the following decomposition of the space: $X = Z \cup \mathcal{T}$, with 
$\mm(Z \setminus \{ f = 0 \}) = 0$ and  $\mathcal{T}$ obtained 
as the disjoint union of a family of timelike geodesics $\{ X_{\alpha}\}_{\alpha \in Q}$
inducing the following disintegration formula: 
$$
\mm\llcorner_{\T} 
= \int_{Q} \mm_{\alpha}\, \qq(d\alpha)= \int_{Q} h(\alpha,\cdot) \, \L^{1}\llcorner_{X_{\alpha}}\, \qq(d\alpha), \qquad \qq \in \mathcal{P}(Q), \quad Q \subset \mathcal{T}.
$$
Moreover, for $\qq$-a.e. $\alpha\in Q$, the following hold:
\begin{itemize}

\item  $h(\alpha,\cdot)\in L^{1}_{loc}(X_{\alpha}, \L^{1}\llcorner_{X_{\alpha}})$; 
\item  $\int_{X_{\alpha}} f \,\mm_{\alpha} = 0$;
\item the one-dimensional m.m.s.
$(X_{\alpha},|\cdot|, \mm_{\alpha})$ satisfies the $\CD(K,N)$ condition, i.e.\;\eqref{eq:DiffIneqCDKN} holds.
\end{itemize}
\end{theorem}

\begin{proof}
By assumption, there exists an $\ell$-cyclically monotone set $\Gamma_{f} \subset X_{\ll}^{2}$ 
and 
$\pi \in \Pi_{\leq}(\mu_{0},\mu_{1})$ that is $\ell$-optimal 
and $\pi(\Gamma_{f}) = 1$. 

We now enlarge $\Gamma_{f}$ 
by filling its possible holes so to restore the regularity we would have if 
$\Gamma_f$ was included in the subdifferential of a Kantorovich potential. 
To this aim, define
$$
\Gamma : = \{ (x,y) \in X^{2}_{\ll} \colon  \tau(z,w) = \tau(z,x) + \tau(x,y)+ \tau(y,w), (z,w) \in \Gamma_{f}\}.
$$
We claim that $\Gamma$ is $\ell$-cyclically monotone; indeed, for $(x_{i},y_{i}) \in \Gamma$ 
\begin{align*}
\sum_{i} \tau(x_{i},y_{i}) 
= &~ \sum_{i} \tau(z_{i},w_{i}) - \tau(z_{i},x_{i}) - \tau(y_{i},w_{i}) \\
\geq &~ \sum_{i} \ell(z_{i+1},w_{i}) -  \tau(z_{i},x_{i}) - \tau(y_{i},w_{i}) \\
\geq &~ \sum_{i} \ell(z_{i+1},x_{i+1}) +\ell(x_{i+1},w_{i})  -  \tau(z_{i},x_{i}) - \tau(y_{i},w_{i}) \\
\geq &~ \sum_{i} \ell(z_{i+1},x_{i+1}) +\ell(x_{i+1},y_{i}) + 
\ell(y_{i},w_{i})   -  \tau(z_{i},x_{i}) - \tau(y_{i},w_{i}) \\
= &~ \sum_{i} \ell(x_{i+1},y_{i}).
\end{align*}
This implies that the pairs belonging to $\Gamma$ are aligned along geodesics: 
if $(x,y) \in \Gamma$ and  $\gamma \in \TGeo(x,y)$ then 
$(\gamma_{s},\gamma_{t}) \in \Gamma$  for all $0\leq s \leq t \leq 1$.
Then we can proceed by defining $R, \T^{end}, \fa(\T^{end}), \fb(\T^{end})$ and $\T$
like few lines before \cref{T:generalDisintegration}. 
Thanks to 
 \cref{T:generalDisintegration}
we obtain 
$$
\mm\llcorner_{\T^{end}} = 
\mm\llcorner_{\T} 
= \int_{Q} \mm_{\alpha}\, \qq(d\alpha)= \int_{Q} h(\alpha,\cdot) \, \L^{1}\llcorner_{X_{\alpha}}\, \qq(d\alpha),
$$
where $\qq$ is a probability measure over the Borel quotient set $Q \subset \T$, 
and for $\qq$-a.e.\;$\alpha$ the one-dimensional metric measure space 
$(X_{\alpha},|\cdot|, \mm_{\alpha})$ satisfies the $\CD(K,N)$ condition.
In particular $\mm_{\alpha}= h(\alpha,\cdot)\mathcal{L}^{1}\llcorner_{X_{\alpha}}$.
We are only left to prove that if $Z = X \setminus \mathcal{T}$, then $f = 0$ $\mm$-a.e.\;over $Z$, 
and that $\int_{Q} f \, \mm_{\alpha} = 0$, for $\qq$-a.e.\;$\alpha$.

Since $\mu_{0}\perp \mu_{1}$, then $\pi(\{ (x,x) \colon x \in X \}) = 0$.
Denoting by $\Delta : = \{ (x,x) \colon x \in X \}$, 
since $\pi(\Gamma) = 1$, we have
$$
\mu_{0}(\mathcal{T}) = \mu_{0} (P_{1}(\Gamma \setminus \Delta)) \geq \pi (\Gamma \setminus \Delta) = 1;
$$
in the same way $\mu_{1}(\mathcal{T}) = 1$.  
This implies that 
$\mu_0(Z) = \mu_1(Z) = 0$. 
Since $\mu_0 = f^+ \mm$ 
and $\mu_1 = f^- \mm$, 
then $Z$ is a subset of $\{ f= 0\}$,  up to a set of $\mm$-measure zero.

\smallskip
We are left to show the balance condition $\int_{Q} f \, \mm_{\alpha} = 0$, for $\qq$-a.e.\;$\alpha$. For any Borel subset $A \subset Q$, consider the saturated set of $A$, 
$R(A) : = P_{2}(A \times X \cap R)$ and compute:
$$
\mu_{0}(R(A)) = \pi ((R(A) \times X) \cap \Gamma ) =  
\pi ( (X \times  R(A)) \cap \Gamma ) = \mu_{1}(R(A)).
$$
Hence for any Borel subset $A \subset Q$ 
$$
\int_{A} \int_{X_{\alpha}} f^{+}\mm_{\alpha}\qq(d\alpha) =
\int_{A} \int_{X_{\alpha}} f^{-}\mm_{\alpha}\qq(d\alpha)
$$
implying that for any Borel subset $A \subset Q$, $\int_{A} \int_{X_{\alpha}} f\, \mm_{\alpha} \qq(d\alpha) = 0$.
The claim is therefore proved.
\end{proof}

Using Theorem \ref{T:localization} it is routine to derive several geometric inequalities in their sharp 
form (following for instance \cite{CM2}). Here we simply report the Brunn-Minkowski inequality.

\begin{proposition}[Sharp timelike Brunn-Minkowski  inequality]\label{prop:BrunnMnk}
Let $(X,\sfd, \mm, \ll, \leq, \tau)$ be a timelike non-branching measured Lorentzian pre-length space satisfying  $\mathsf{TCD}^{e}_{p}(K,N)$, for some $K\in \R, N\in [1,\infty), p\in (0,1)$.
\\Let $A_{0}, A_{1}\subset X$ be measurable subsets with $\mm(A_{0}), \mm(A_{1})\in (0,\infty)$. Calling 
$\mu_{i}:=1/\mm(A_{i}) \, \mm\llcorner_{A_{i}}$, $i=1,2$, assume that $(\mu_{0},\mu_{1})$ 
is  timelike $1$-dualisable. 

Then
$$
\mm(A_{t})^{1/N}\geq \tau^{(1-t)}_{K,N} (\theta)\, \mm(A_{0})^{1/N} 
+  \tau^{(t)}_{K,N} (\theta ) \, \mm(A_{1})^{1/N}
$$
where  $A_{t}:=\cI(A_{0},A_{1},t)$ defined in \eqref{eq:defI(A,B,t)} is the set of $t$-intermediate points of geodesics from $A_{0}$ to $A_{1}$, and $\Theta$ is the maximal/minimal time-separation between points in $A_{0}$ and $A_{1}$, i.e.:
\begin{equation*}
\theta:=
\begin{cases}
\sup\{\tau(x_{0}, x_{1}): x_{0}\in A_{0}, x_{1}\in A_{1}\} & \text{if } K< 0,\\
\inf\{\tau(x_{0}, x_{1}): x_{0}\in A_{0}, x_{1}\in A_{1}\} \quad &\text{if } K\geq 0 .
\end{cases}
\end{equation*}
Finally $\tau^{(t)}_{K/N} (\theta ) : = t^{\frac{1}{N}}\sigma_{K/N}^{(t)}(\theta)^{\frac{N-1}{N}}$, where the volume distortion coefficients $\sigma_{K/N}^{(t)}(\theta)$ are defined in \eqref{eq:sigmakappa}.
\end{proposition}

Once \cref{T:localization} is at disposal,  \cref{prop:BrunnMnk} can be proved following verbatim the proof of \cite[Thm.\;3.1]{CM2}.

\footnotesize{
\bibliographystyle{plain}
\bibliography{literature.bib}
}
\end{document}